\documentclass[11pt]{article}
	
	\newcommand{\blind}{1}
	
    \makeatletter
    \renewcommand\section{\@startsection {section}{1}{\z@}%
    {-3.5ex \@plus -1ex \@minus -.2ex}%
    {2.3ex \@plus.2ex}%
    {\normalfont\fontfamily{phv}\fontsize{16}{19}\bfseries}}
    \renewcommand\subsection{\@startsection{subsection}{2}{\z@}%
    {-3.25ex\@plus -1ex \@minus -.2ex}%
    {1.5ex \@plus .2ex}%
    {\normalfont\fontfamily{phv}\fontsize{14}{17}\bfseries}}
    \renewcommand\subsubsection{\@startsection{subsubsection}{3}{\z@}%
    {-3.25ex\@plus -1ex \@minus -.2ex}%
    {1.5ex \@plus .2ex}%
    {\normalfont\normalsize\fontfamily{phv}\fontsize{14}{17}\selectfont}}
    \makeatother
	
	\usepackage{amsmath}
	\usepackage{enumerate}
	\usepackage{xcolor}
    
	\usepackage{natbib} 
	\usepackage{url} 
 \usepackage{amsmath, amsthm, amssymb}
\usepackage[english]{babel}
\newtheorem{prop}{Proposition}

\usepackage[ruled,linesnumbered]{algorithm2e}
\usepackage{ragged2e}
\usepackage{xcolor}
\usepackage{setspace}
\allowdisplaybreaks
\usepackage{fix-cm}
\usepackage{threeparttable}
\usepackage{appendix}
\usepackage{tabularx}
\usepackage{graphicx}
\usepackage{float} 
\usepackage{tikz, tikz-3dplot}
\usepackage[skip=2pt,font=scriptsize]{caption}

\usepackage{abstract}
\usepackage{caption}
\usepackage{amsthm}
\usepackage{comment}
\usepackage{subcaption}
\usepackage{mathtools} 
\usepackage{comment}
\usepackage{makecell}
\usepackage{algorithm2e}
\usepackage{algpseudocode}
\usepackage{enumitem}

\RestyleAlgo{ruled}

\definecolor{links}{RGB}{204,36,29}
\usepackage[colorlinks=true,breaklinks=true,bookmarks=true,urlcolor=links,citecolor=links,linkcolor=links,bookmarksopen=false,draft=false]{hyperref}
	
\usepackage[figuresright]{rotating}
\RequirePackage{pdflscape}
\usepackage{multirow, booktabs}
	
\begin{document}
		
		\def\spacingset#1{\renewcommand{\baselinestretch}%
			{#1}\small\normalsize} \spacingset{1}
		
		\if0\blind
		{
			\title{Joint Bundle Design and Pricing for Extended Warranty Providers Servicing Multi-Tier Products}
            \author{Author information is purposely removed for double-blind review}
			
\bigskip
			\bigskip
			
			\medskip
			\date{}
			\maketitle
		} \fi
		
		\if1\blind
		{

            \title{Joint Bundle Design and Pricing for Extended Warranty Providers Servicing Multi-Tier Products}
			\author{Yajing Chen$^a$, 
            Yanrong Li$^b$, Xiao-Lin Wang$^c$, Zhi-Sheng Ye$^b$ \\
			$^a$ School of Management Science and Engineering, Dongbei University of Finance \\
            and Economics, Dalian 116025, China \\
            $^b$ Department of Industrial Systems Engineering and Management, \\
            National University of Singapore, Singapore \\
            $^c$ Business School, Sichuan University, Chengdu 610065, China\\
            }
			
\bigskip
			\bigskip
			\bigskip
            \maketitle
			\medskip
		} \fi
		\bigskip
		
\begin{abstract}
Extended warranties (EWs) constitute a significant source of revenue for capital-intensive products. Such products comprise multiple subsystems, enabling flexible EW design. For example, providers can bundle tailored sets of subsystems within different EW contracts, facilitating the creation of a service menu with differentiated warranty options.
From the perspective of a third-party EW provider servicing multi-tier products, we develop a novel model to jointly optimize bundle design and pricing for EW options in order to maximize the expected total profit.
Specifically, the problem involves determining which contracts\textemdash each containing a differentiated bundle of subsystems\textemdash to recommend for the multi-tier products and identifying the appropriate price for each contract. As the complexity of the joint optimization problem increases exponentially with the number of subsystems, we devise two solution approaches. The first approach leverages a mixed-integer second-order cone programming reformulation, which guarantees optimality but is applicable only for a small number of subsystems. 
The second approach utilizes an iterative two-step process, offering enhanced computational efficiency for scenarios involving a large number of subsystems. Numerical experiments validate the effectiveness of our model, particularly in scenarios characterized by high failure probabilities and a large number of subsystems.

\end{abstract}
			
\noindent%
\spacingset{1.5} 

\section{Introduction} \label{s:intro}
An extended warranty (EW) provides prolonged protection against product failures beyond the manufacturer's base warranty. Under an EW contract, the provider promises to investigate, assess, and provide compensation if any failure covered in the contract occurs during the protection period \citep{murthy2014extended}. Unlike base warranty that is an integrated part of product sales and free of charge to customers, EW has to be bought separately with extra premiums, thus becoming an essential revenue source for providers. Recent years have witnessed substantial growth in the EW market. According to the report by \cite{imarc2025}, the global EW market was valued at \$147.1 billion in the year 2024 and is expected to reach \$239.0 billion by 2033.
In the EW market, capital-intensive products are particularly prominent due to their high values and extended lifespans. 
In practice, the base warranties of most capital-intensive products only cover a limited use time and/or usage (e.g., 3 years or 36,000 miles for automobiles), leaving the products unprotected after the base warranties expire. In view of this business opportunity, many providers start to offer EW contracts to satisfy customers' demand of post-warranty protection against product failures.

Unlike base warranties regulated by relevant policies, EWs offer greater flexibility in terms of design and pricing, allowing providers to offer more personalized contracts.
For instance, an EW is allowed to cover a partial set of subsystems in a product, whereas a base warranty generally covers all critical subsystems as required by most regulation policies. In this sense, providers can design multiple EW options by bundling differentiated sets of subsystems. 
As an example, Ford offers four EW options for its automobiles, including PowertrainCARE, BaseCARE, ExtraCARE, and PremiumCARE, which cover 29, 84, 113, and 1000+ components, respectively \citep{Ford}.
In recent years, due to the high profitability of EWs, many third-party EW providers have emerged in the aftermarket to provide services for products made by different manufacturers and of different functionalities, price bands, and/or base-warranty statuses, among others. 
For example, Lizhen Company, a leading third-party automobile warranty service provider in China, prescribes the composition of subsystems in their EW contracts based on the power source and the status of base warranty \citep{Lizhen}.
Specifically, the company roughly categorizes automobiles into three series and offers differentiated EWs: (1) automobiles under base warranties (T1-T3 series) are offered EW coverage for fixed 16 or 17 major subsystems; (2) automobiles with expired base warranties (T4T5 series) are offered EW coverage for fixed 3 or 4 core subsystems and (3) new energy series are offered EW coverage only for essential electrical systems; see Figure~\ref{Fig:1-EW} for details. 
Nevertheless, this strategy provides the same bundle of subsystems for all automobiles in the same series, with the warranty price depending solely on the automobiles' price bands. 
Similarly, third-party EWs in other industries (e.g., construction equipment, home appliances) also fail to capture the diversity and heterogeneity of subsystems among different products.


\begin{figure}[t]
    \centering
    \includegraphics[width=0.78\linewidth]{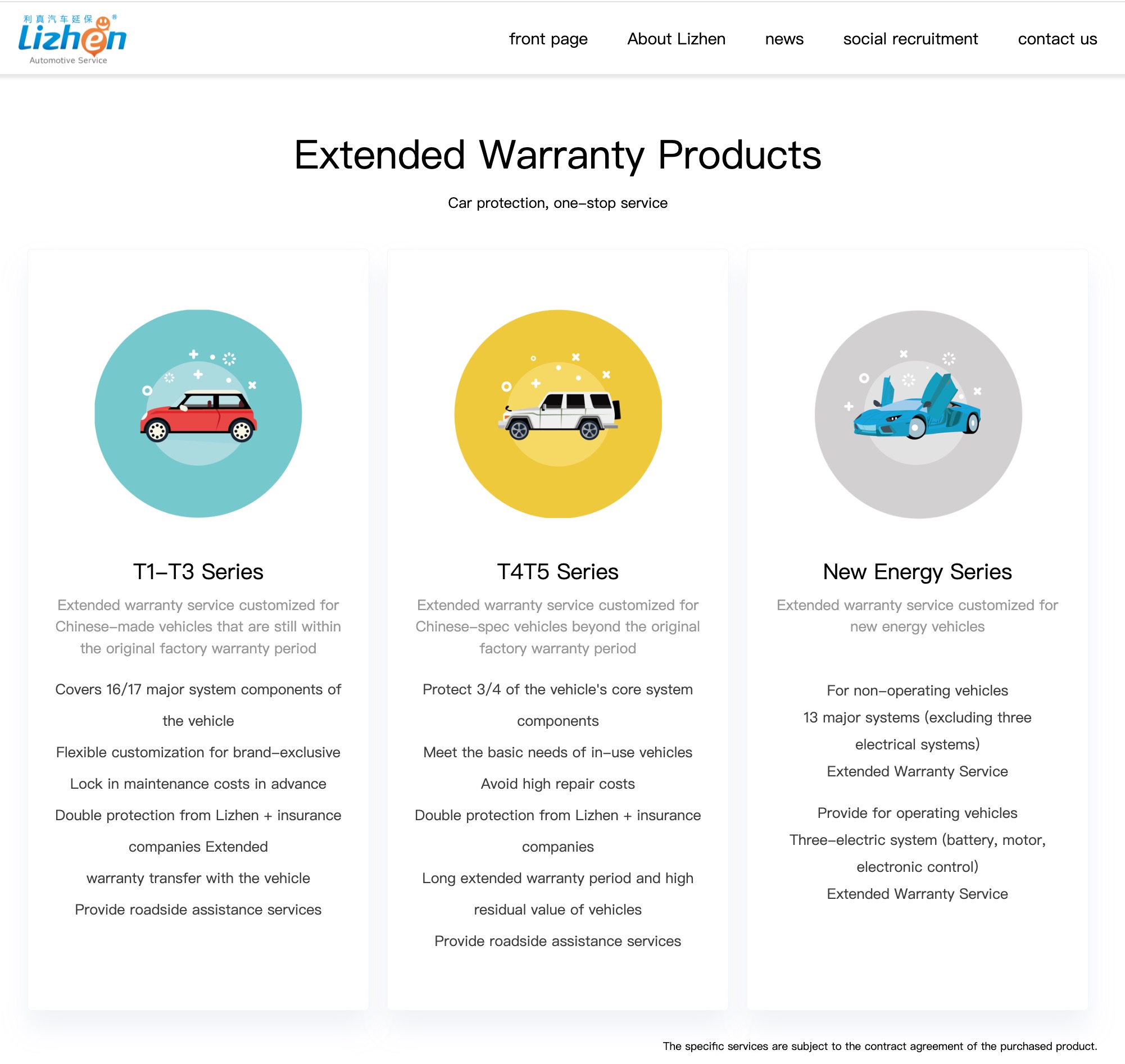}
    \caption{\centering Current EWs in the aftermarket offered by Lizhen Company}
    \label{Fig:1-EW}
\end{figure}

In the literature, extensive research efforts have been devoted to the design and pricing of flexible EW menus \citep{ye2013optimal, xie2017two, wang2020design, wang2023design, wang2025design}, 
particularly emphasizing how to specify protection periods for different warranty options and how to price them accordingly.
Although such research efforts have significantly contributed to maximizing profits or minimizing costs for EW providers, two main limitations still exist. 
First, they typically concentrate on a single product brand from the manufacturer’s perspective, which cannot be directly applied to a third-party EW provider. Due to the extensive range of product brands and models that third-party providers must service, designing personalized services for each product would substantially increase marketing and promotion costs. 
Second, the EW options in most studies differ only in their protection lengths and prices, but are homogeneous in terms of coverage\textemdash that is, the number of subsystems or services covered \citep[see, e.g.,][]{wang2021design, wang2023design,dai2024optimizing}. Such design does not fully reflect the flexibility of EWs, as it overlooks the difference among various subsystems.

In this paper, we develop a novel method to jointly optimize bundle design and pricing for EW contracts from a third-party EW provider's perspective.
We propose to categorize the products serviced by the EW provider into different groups based on their functionalities, power sources, base-warranty statuses (i.e., expired or not), and/or price bands (e.g., luxury, commercial, and economy), thereby simplifying the groups of products and reducing promotion costs. 
On this basis, we design an EW menu that consists of multiple differentiated contracts, each tailored to a specific group of products.
The basic attraction model is adopted to describe customers' choice probabilities. Our main focus is on identifying which subsystems should be included in an EW contract, determining the most suitable contracts for each group of products, and specifying an optimal EW price when bundling different subsystems within each contract. In doing so, a bundle pricing approach is employed to determine optimal prices for EW contracts covering different sets of subsystems.

Unfortunately, when fully considering these factors, the joint optimization problem becomes intractable due to the following reasons. First, the potential combinations of subsystems in EW contracts grow exponentially as the number of subsystems increases. Even though the number of contracts in an EW menu is limited (usually 2 to 5), the specification of optimal compositions for the contracts is still not an easy task. Second, the warranty price is another critical decision variable that affects customers' preference weight and the provider's total profit; jointly optimizing the compositions and prices of EW contracts further increases the difficulty of problem-solving. Finally, the diversity of customers who own products in different groups leads to variations in their willingness to pay, adding another complication to the optimization process.

Our work contributes to the literature in the following respects:
(1) To the best of our knowledge, this work is the first attempt to jointly optimize the compositions and prices of EW contracts for multi-tier products, which is of particular interest to third-party EW providers. 
(2) We theoretically show that the joint optimization problem is NP-hard. 
As the complexity of this problem increases exponentially with the number of subsystems, we propose two solution methods: a mixed-integer second-order cone programming (MISOCP) reformulation method that ensures optimality but only works for small-scale problems, and an iterative two-step (ITS) method that can enhance computational efficiency for large-scale problems.
(3) We conduct extensive numerical experiments to verify the proposed two solution approaches. Additionally, we demonstrate the advantages of the joint optimization model over other benchmark models currently used in practice. The impacts of different model parameters on the optimal results are analyzed and the corresponding managerial insights are discussed.

The remainder of this article is organized as follows. Section \ref{Sec:Liter} reviews related literature. Section \ref{Sec:Model} presents the joint EW design and pricing model for multi-tier products from the third-party provider's perspective. Section \ref{Sec:solu} proposes two approaches to solving the joint optimization model. Section \ref{Sec:Numerial} presents numerical results for model comparisons and sensitivity analyses. 
The managerial insights are also discussed. Section \ref{Sec:con} summarizes this paper and points out some future research topics. The proofs and additional technical details can be found in Supplemental Online Materials.

\section{Literature Review}\label{Sec:Liter}

Our work is related to two streams of literature: (i) design and pricing of EW contracts and (ii) bundle pricing.

Given the critical role of EW services in prolonging product lifecycle, strengthening customer retention, and creating sustained profits, extensive literature has focused on the design and optimization of EW strategies.
For example, \cite{su2012analysis} studied EW policies under different repair options (namely, minimal, imperfect, and complete) from the manufacturer's perspective. 
\cite{li2023warranty} designed an EW contract for deteriorating products with maintenance duration commitments aiming to maximize the manufacturer’s profits. 
In addition, recent studies have begun to jointly optimize EW along with product and warranty sales strategies. 
\cite{qiao2024optimal} proposed an integrated model of warranty and post-warranty/EW maintenance from the consumers' perspective, encouraging consumers to consider warranty plans and EW maintenance strategies simultaneously when making a purchase decision. 
\cite{wang2024simultaneous} studied simultaneous and sequential retail strategies for a durable product and its EW, and examined the impact of the two retail strategies on the firm's optimal pricing decisions and profitability.

In particular, our work builds on the existing literature on the design and pricing of flexible EW contracts. \cite{hartman2009designing} proposed a flexible warranty menu for customers, evaluated their individual optimal choice using dynamic programming, and priced the EW service for each customer. \cite{liu2020complimentary} proposed a flexible post-purchase warranty choice of online registration for customers and optimized the warranty price with heterogeneous customer risk attitudes. \cite{wang2021design} designed a customized EW according to heterogeneous customer usage rates, in which a customer is allowed to choose the age limit of EW based on his/her willingness to pay. 
\cite{wang2020design,wang2025design}, \cite{wang2023design}, and \cite{dai2024optimizing} studied the design and pricing of various EW menus based on the multinomial logit model, where the menus consist of multiple EW options with differentiated features such as protection lengths and prices. 
However, existing EW design and pricing studies focus predominantly on specifying optimal protection lengths and/or prices for EW contracts, provided that the contracts have a homogeneous composition of subsystems. 
Our work distinguishes itself from the existing research by determining optimal compositions of subsystems and the corresponding EW prices, which reflect the priority concern of third-party EW providers.


Furthermore, our paper is also related to the bundle pricing literature. This stream of research can be traced back to the 1960s, when early studies demonstrated that bundling could significantly capture consumer surplus, reduce sales costs, and minimize consumer heterogeneity \citep{stigler1963united,adams1976commodity,schmalensee1984gaussian}.
\cite{guiltinan1987price} proposed a bundle pricing strategy of two or more products/services at a special discount, which extended the economic theory of bundling to encompass various types of complementary relationships.
\cite{hanson1990optimal} considered multiple components and obtained optimal bundling prices by mixed integer programming.
\cite{bitran2007pricing} addressed a problem of determining optimal composition and price for high-tech manufacturers by separating the two decisions: optimizing price through derivative analysis and optimizing composition based on a revenue-by-order criterion.
Recent research also considered other important factors such as customer behaviors. For example, \cite{tulabandhula2023multi} proposed a bundle multivariate logit model to characterize multi-purchase behaviors and optimized recommendations for online retailers and e-commerce platforms. 
\cite{nie2024sales} discussed the advantages of bundle and add-on sales strategies using game theory, considering customer returns in e-commerce. 
\cite{sun2025partition} studied bundling and pricing strategies when the new offering interacts with and potentially cannibalizes existing product offerings.
However, the bundle pricing strategy has not yet been applied to EWs for multi-tier products serviced by third-party EW providers. We contribute to the literature by personalizing the compositions\textemdash in terms of subsystems covered\textemdash of EW contracts for multi-tier products, and implementing optimized bundle pricing strategies to maximize the provider's expected total profit.

\section{Joint Optimization Model}\label{Sec:Model}
In this section, an optimization model is developed for joint bundle design and pricing of EW contracts for multi-tier products from the perspective of a third-party warranty provider. Section~\ref{Subsec:setting} first describes the problem setting. Section~\ref{Subsec:Joint} then presents the joint optimization model. 
The notations used throughout this paper are summarized in Table~\ref{notation}.
Sets, vectors, and matrices are denoted by uppercase calligraphy letters, lowercase bold letters, and uppercase bold letters, such as $\mathcal{Y}$, $\mathbf{y}$, and $\mathbf{Y}$, respectively.

\begin{table}[ht!]
	\centering
	\fontsize{8}{10}\selectfont
	\caption{Notations}
	\label{notation}
	{\renewcommand{\arraystretch}{1.5} 
		\begin{tabular}{ll}
			\hline 
			\multicolumn{2}{l}{\textbf{Sets}}  \\
			\hline$\mathcal{M}$ & Set of groups of products \\
			$\mathcal{W}$ & Set of subsystems in the products \\
			$\mathcal{N}$ & Set of EW contracts, with $n=2^{w}-1$ \\
			$\mathcal{L}$ & Set of discount ladders\\
			\hline 
			\multicolumn{2}{l}{\textbf{Indices}}  \\
			\hline $i\in \mathcal{N}$ & Index of EW contract   \\
			$j\in \mathcal{M}$ & Index of product groups \\
			$k\in \mathcal{W}$ & Index of subsystems \\
			$h\in \mathcal{L}$ & Index of discount levels\\
			\hline 
			\multicolumn{2}{l}{\textbf{Parameters}}  \\
			\hline 
			$d_h$ & Discount level of the $h$th ladder \\
			$v_{kj}$ & Attraction value of subsystem $k$ for group-$j$ customers \\
			$u_{0j}$ & Preference weight of the outside option for group-$j$ customers \\
			$P_{kj}^0$ & Initial EW price for subsystem $k$ in products belonging to group-$j$ customers \\
			$\beta_{j}$ & Price sensitivity of group-$j$ customers \\
			$\mathbb{I}_{i k}$ & Binary indicator indicating whether contract $i$ includes subsystem $k$ \\
			$\lambda_j$ & The proportion of group-$j$ customers\\
			$F_{kj}$ & Failure probability of subsystem $k$ in products belonging to group-$j$ customers\\
			$c_{kj}$ & Failure cost of subsystem $k$ in products belonging to group-$j$ customers \\
			$\theta$ & Advertising cost of an EW \\
			\hline 
			\multicolumn{2}{l}{\textbf{Decision variables}}  \\
			\hline $p_{ij}$ & Selling price of EW contract $i$ for group-$j$ customers \\
			$x_{ij}$ & Binary variable indicating whether EW contract $i$ is offered to group-$j$ customers  \\
			$y_{i}$ & Binary variable indicating whether EW contract $i$ is in the advertising list \\
			$z_{ih}$ & Binary variable indicating whether the discount of contract $i$ is $d_h$\\
			\hline   
	\end{tabular}}
\end{table}

\subsection{Problem Setting}\label{Subsec:setting}
Consider a third-party EW provider aiming to accommodate multiple brands of capital-intensive products in the aftermarket. To simplify management, it is necessary to implement a systematic classification of these brands prior to EW design. Specifically, products are categorized into $m$ groups, indexed by $j \in \mathcal{M} := \{1, 2, \dots, m\}$.
A customer who owns a product in the $j$th group is called a group-$j$ customer.
The proportion of group-$j$ customers is denoted by $\lambda_j$, with $\sum_{j\in \mathcal{M}}\lambda_j=1$.
Suppose that there are $w$ subsystems (e.g., engine systems and suspension systems for automobiles), indexed by $k\in \mathcal{W}:=\{1,2,\dots,w\}$, that are critical to the functioning of the products. 
In total, there are at most $n=2^{w}-1$ combinations of subsystems, each composing a potential EW contract, which are indexed by $i\in \mathcal{N}:=\{1,2,\dots,n\}$.
We use the binary indicator $\mathbb{I}_{ik}$ to indicate whether subsystem $k$ is included in EW contract $i$. That is, if subsystem $k$ is present in contract $i$, then $\mathbb{I}_{ik}=1$; otherwise, $\mathbb{I}_{ik}=0$. 
The provider must decide on an appropriate EW contract for each customer group. To this end, we use the binary variable $x_{ij}$ to indicate whether contract $i$ should be recommended to group-$j$ customers. That is, if contract $i$ is recommended to group-$j$ customers, then $x_{ij}=1$; otherwise, $x_{ij}=0$.

In this work, we consider two cost elements that are incurred to the provider when managing the EW business: warranty cost for servicing customer claims and advertising cost for promoting EWs. 
We adopt the replacement policy to handle warranty claims \citep{chien2015analysis,jin2024allocating,liu2023generalized}; that is, the EW provider is obligated to replace the failed systems with new identical ones free of charge to customers.
Let $F_{kj}$ represent the average failure probability of subsystem $k$ in group-$j$ products over the EW period, and $c_{kj}$ the associated average replacement cost.
Moreover, since the provider can consolidate the advertising activity for contracts designed for different groups of products as long as the contracts cover the same subsystem combination, we use the binary variable $y_i$ to denote whether contract $i$ should be advertised (namely, included in the displayed menu, as shown in Figure~\ref{Fig:1-EW}). 
That is, as long as contract $i$ is recommended for any product in group $j$ (i.e., $\exists$ $j \in \mathcal{M}$ \(|\) $x_{ij}=1$), $y_i$ should be set to 1, and the marginal advertising cost $\theta$ should be incurred to the EW provider.

In our context, an EW contract usually covers multiple subsystems. As a result, the warranty price should depend on the covered subsystems and the price brand of the products.
Let $p_{ij}$ denote the price of EW contract $i$ for products in group $j$. The pricing mechanism is as follows. 
For each group $j$, an initial price $P_{kj}^0$ is predefined for each subsystem $k$, which is proportional to the corresponding failure cost $c_{kj}$. Then, we adopt the pricing strategy proposed by \cite{guiltinan1987price}, which is a normative framework for optimizing the different mixed-bundling discounts for the different subsystems included. Compared with an equivalent price promotion strategy, the bundle discount-based strategy shows advantages in attracting consumers \citep{janiszewski2004influence} and enhancing their loyalty in a competitive market \citep{balachander2010bundle}. 
In line with the discrete price ladder practice \citep{cohen2021promotion}, the bundle discount is usually restricted to a finite set of admissible values. Suppose that there are $l$ admissible discount levels, denoted by $\{d_h\}_{h \in \mathcal{L}}$, $\mathcal{L}:=\{1,2,\dots, l\}$. Without loss of generality, the discounts are indexed such that
\[
d_1>d_2>\dots>d_l,
\]
where a smaller value of $d_h$ ($h \in \mathcal{L}$) indicates a larger discount.
We use the binary variable $z_{ih}$ to indicate whether discount level $d_h$ should be assigned to contract $i$. If discount $d_h$ is selected for contract $i$, then $z_{ih}=1$; otherwise, $z_{ih}=0$. 
In this manner, the final price of EW contract $i$ advertised to group-$j$ customers is expressed as
\begin{equation}\label{Equ:EWprice}
    p_{ij}=\left(\sum_{k\in \mathcal{W}}\mathbb{I}_{ik}P_{kj}^0 \right)\left(\sum_{h\in \mathcal{L}}z_{ih}d_{h} \right), 
\end{equation}
with $\sum_{h \in \mathcal{L}} z_{i h}=1$, implying that only one discount level is assigned to each contract $i$.

Suppose that customers are rational in the sense that they make the purchase choices to maximize their utilities. 
Let $u_{ij}$ denote the preference weight of EW contract $i$ for group-$j$ customers, which is a function of the average valuation and price of the contract, and thus expressed as $u_{ij}(p_{ij})$.
Further let $u_{0j}$ denote the preference weight of the outside option for group-$j$ customers. 
According to the demand setting in \cite{luce1959individual} and \cite{hu2025constant}, we employ the basic attraction model to describe customers' choice among available EW contracts. Under this model, the probability that a group-$j$ customer chooses contract $i$ is given by
\begin{equation} \label{choice prob_0}
  q_{i j}\left(\mathbf{u}_{j}, \mathbf{x}_{j}\right)
  = \frac{u_{ij}(p_{ij})}{u_{0 j} + \sum_{i' \in \mathcal{N}} [u_{i'j}(p_{i'j})\cdot x_{i' j}]} x_{i j},
\end{equation}
where $\mathbf{x}_j = \{x_{ij}\}_{i \in \mathcal{N}}$ and $\mathbf{u}_j = \{u_{ij}(p_{ij})\}_{i \in \mathcal{N}}$.

For notational convenience, we further express the decision variables in vector or matrix form:
$\mathbf{p}_j = \{p_{ij}\}_{i \in \mathcal{N}}$,
$\mathbf{P}=\{\mathbf{p}_j\}_{j\in\mathcal{M}}$,
$\mathbf{X}=\{\mathbf{x}_j\}_{j\in\mathcal{M}}$, 
$\mathbf{y} = \{y_{i}\}_{i \in \mathcal{N}}$, 
$\mathbf{Z} = \{z_{ih}\}_{i \in \mathcal{N}, h \in \mathcal{L}}$.

\subsection{Joint Bundle Design and Pricing Model}\label{Subsec:Joint}
Based on the setting described previously, we develop a mixed integer programming model to seek optimal compositions and prices of EWs in order to maximize the expected total profit from a third-party EW provider's perspective. 

Once a group-$j$ customer chooses EW contract $i$, the marginal profit for the EW provider can be calculated as the difference between the sales revenue and the warranty cost (i.e., $p_{ij}-\sum_{k \in \mathcal{W}} \mathbb{I}_{i k} F_{k j} c_{kj}$). Based on the marginal profit and the choice probability of each EW contract, the expected total profit for the EW provider can be evaluated by Equation (\ref{mainObject}) below. Then, the joint EW design and pricing problem (\texttt{JDPEW}) is defined as
\begin{align}
  \max \,\ & \sum_{j \in \mathcal{M}} \lambda_j \sum_{i \in \mathcal{N}} q_{i j}\left(\mathbf{u}_{j}, \mathbf{x}_{j}\right) \left[p_{ij}-\sum_{k \in \mathcal{W}} \mathbb{I}_{i k} F_{k j} c_{kj}\right]-\sum_{i \in \mathcal{N}} \theta y_i \label{mainObject}\\ 
  \text{s.t.} \,\ 	 
  &x_{i j} \leq y_i  && \forall i \in \mathcal{N},~j \in \mathcal{M} \label{constraint1}\\
  &\sum_{i \in \mathcal{N}} \mathbb{I}_{i k} x_{i j} \geq 1 \quad && \forall k \in \mathcal{W},~j \in \mathcal{M}\label{constraint2}\\
  &\sum_{h \in \mathcal{L}} z_{i h}=1 &&  \forall i \in  \mathcal{N}\label{constraint3} \\
  &\left(\sum_{h \in \mathcal{L}} z_{i h} d_h-\sum_{h^{\prime} \in \mathcal{L}} z_{i^{\prime} h^{\prime}} d_{h^{\prime}}\right)\left(\sum_{k \in \mathcal{W}} \mathbb{I}_{i k}-\sum_{k \in \mathcal{W}} \mathbb{I}_{i^{\prime} k}\right) \leq 0 &&  \forall i,~i' \in \mathcal{N}\label{constraint4} \\
  &p_{ij}=\left(\sum_{k\in \mathcal{W}}\mathbb{I}_{ik}P_{kj}^0\right)\left(\sum_{h\in \mathcal{L}}z_{ih}d_{h}\right)&& \forall i \in \mathcal{N},~j \in \mathcal{M} \label{constraint5}\\
  &y_i \in\{0,1\},\, x_{i j} \in\{0,1\},\, z_{i h} \in\{0,1\},\, p_{ij}\geq 0 && \forall i \in \mathcal{N},~j \in \mathcal{M},~h \in \mathcal{L}.\label{constraint6}
\end{align}
The objective is to maximize the expected total profit in Equation~\eqref{mainObject}, calculated as the expected revenues minus the servicing and advertising costs for all the EWs. 
Constraint~\eqref{constraint1} ensures that the recommended EW contract for each group of products must be advertised to the associated customers. 
Constraint~\eqref{constraint2} states that for each group of products, every subsystem must appear at least once in the final recommended EWs, because the concerned subsystems are critical to the functionality of the products. In this manner, we guarantee that any subsystem of each group of products can be included in the EW contract to accommodate customers' specific preferences for particular subsystems.
Constraint~\eqref{constraint3} ensures that each EW contract has only one discount level.
Constraint~\eqref{constraint4} indicates that the EW provider offers a larger discount to customers who purchase EWs involving a larger number of subsystems, so as to stimulate demand. In other words, the more subsystems an EW contract includes, the larger discount it has. Constraint~\eqref{constraint5} represents the bundle pricing mechanism.

As discussed earlier, the provider can jointly promote multiple EW contracts as long as they cover identical subsystems across multiple groups of products so as to reduce advertising costs. The resultant advertising cost $\sum_{i\in\mathcal{N}} \theta y_i$ in Equation~\eqref{mainObject} cannot be decomposed by product group $j$. As a result, the optimization problem (\texttt{JDPEW}) cannot be decomposed into several smaller-scale subproblems, each focusing on the EW design and pricing for only one group of products.
In essence, problem (\texttt{JDPEW}) is a nonlinear program with a large number of decision variables and constraints, which poses significant challenges for problem solving. The following proposition reveals the hardness of this problem. 

\begin{prop}[Hardness]\label{Prop1}
 Problem (\texttt{JDPEW}) is NP-hard.  
\end{prop}

To solve this NP-hard problem, we propose two tractable approaches in the next section. 

\section{Solution Approaches}\label{Sec:solu}
In this section, we focus on the solution approaches for problem (\texttt{JDPEW}). Before proceeding, we first present a specific expression of the demand model where the attraction value of each contract is a linear function of the average valuation and price of the contract.

Let $v_{kj}$ represent the average valuation stemming from the protection for subsystem $k$ perceived by group-$j$ customers. 
According to \cite{fader1996modeling}, \cite{benson2018discrete}, and \cite{tulabandhula2023multi}, the attraction values for different subsystems are assumed to be additive (this setting is relaxed in Appendix \ref{Sec:Rob}), and thus the gross attraction value of contract $i$ perceived by group-$j$ customers can be expressed as $\sum_{k \in \mathcal{W}} \mathbb{I}_{i k} v_{k j}$.
Further let $\beta_j>0$ denote the price sensitivity of group-$j$ customers. Then, the preference weight $u_{ij}(p_{ij})$ can be expressed as
\begin{equation}\label{Equ:utility}
	u_{ij}(p_{ij})=\sum_{k \in \mathcal{W}} \mathbb{I}_{i k} v_{k j}-\beta_jp_{ij}.
\end{equation}		 
By substituting Equation~\eqref{Equ:utility} into Equation~\eqref{choice prob_0}, the probability that a group-$j$ customer chooses contract $i$ can be rewritten as
\begin{equation}\label{choice prob}
\begin{aligned}
	q_{i j}\left(\mathbf{p}_{j}, \mathbf{x}_{j}\right)
	= \frac{\sum_{k \in \mathcal{W}} \mathbb{I}_{i k} v_{k j}-\beta_j p_{ij}}{u_{0 j}+\sum_{i' \in \mathcal{N}}\left[\sum_{k \in \mathcal{W}} \mathbb{I}_{i' k} v_{k j}-\beta_j p_{i'j} \right] x_{i'j}} x_{i j}.
\end{aligned}	
\end{equation}
Our subsequent analysis is based on this particular form of the basic attraction model.

For ease of presentation, we define $\hat{V}_{i j}=\sum_{k \in \mathcal{W}} \mathbb{I}_{i k} v_{k j}$, $\hat{P}_{i j}=\sum_{k \in \mathcal{W}} \mathbb{I}_{i k} P^{0}_{k j}$, and $\hat{C}_{i j}=\sum_{k \in \mathcal{W}} \mathbb{I}_{i k} F_{k j} c_{kj}$. Then, substituting Equation~(\ref{constraint5}) into the objective function  of problem (\texttt{JDPEW}) in Equation~(\ref{mainObject}) yields 
\begin{equation}
  \sum_{j \in \mathcal{M}} \lambda_j \frac{\sum_{i \in \mathcal{N}}\left[\hat{V}_{i j}-\beta_j \hat{P}_{i j}\left(\sum_{h \in \mathcal{L}} z_{i h} d_h\right)\right]\left[\hat{P}_{i j}\left(\sum_{h \in \mathcal{L}} z_{i h} d_h\right)-\hat{C}_{i j}\right] x_{i j}}{u_{0 j}+\sum_{i' \in \mathcal{N}}\left[\hat{V}_{i' j}-\beta_j \hat{P}_{i' j}\left(\sum_{h \in \mathcal{L}} z_{i' h} d_h\right)\right] x_{i' j}}-\sum_{i \in \mathcal{N}} \theta y_i. \label{reformulation:obj1}
\end{equation}

Based on this new formulation, we propose two solution approaches\textemdash that is, mixed-integer second-order cone programming (MISOCP) reformulation and iterative two-step (ITS) approaches in Sections~\ref{Subsec:socp} and \ref{Subsec:twostage}, respectively, and analyze their properties with respect to optimality and computational complexity.

\subsection{MISOCP Reformulation} \label{Subsec:socp}
Letting $\hat{R}_j=\max _{i \in \mathcal{N}}\{\hat{P}_{i j}-\hat{C}_{i j}\}$, the objective function in \eqref{reformulation:obj1} can be rewritten as
\begin{equation} \footnotesize
   \sum_{j \in \mathcal{M}} \lambda_j \hat{R}_j -\sum_{j \in \mathcal{M}} \lambda_j \frac{u_{0 j} \hat{R}_j+\sum_{i \in \mathcal{N}}\left[\hat{V}_{i j}-\beta_j \hat{P}_{i j}\left(\sum_{h \in \mathcal{L}} z_{i h} d_h\right)\right]\left[\hat{R}_j-\hat{P}_{i j}\left(\sum_{h \in \mathcal{L}} z_{i h} d_h\right)+\hat{C}_{i j}\right] x_{i j}}{u_{0 j}+\sum_{i' \in \mathcal{N}}\left[\hat{V}_{i' j}-\beta_j \hat{P}_{i' j}\left(\sum_{h \in \mathcal{L}} z_{i' h} d_h\right)\right] x_{i' j}}- \sum_{i\in N}\theta y_i. \label{reformulaton:obj2}
\end{equation}
Because the term $\sum_{j \in \mathcal{M}} \lambda_j \hat{R}_j$ is a constant, maximizing the expected total profit given by Equation~\eqref{reformulaton:obj2} is equivalent to minimizing 
\begin{equation}\footnotesize
  \sum_{j \in \mathcal{M}} \lambda_j \frac{u_{0 j} \hat{R}_j+\sum_{i \in \mathcal{N}}\left[\hat{V}_{i j}-\beta_j \hat{P}_{i j}\left(\sum_{h \in \mathcal{L}} z_{i h} d_h\right)\right]\left[\hat{R}_j-\hat{P}_{i j}\left(\sum_{h \in \mathcal{L}} z_{i h} d_h\right)+\hat{C}_{i j}\right] x_{i j}}{u_{0 j}+\sum_{i' \in \mathcal{N}}\left[\hat{V}_{i' j}-\beta_j \hat{P}_{i' j}\left(\sum_{h \in \mathcal{L}} z_{i' h} d_h\right)\right] x_{i' j}}+ \sum_{i\in N}\theta y_i. 
\end{equation}
Let 
\begin{equation*}
   t_j=\frac{1}{u_{ 0j}+\sum_{i \in \mathcal{N}}\left[\hat{V}_{i j}-\beta_j \hat{P}_{i j}\left(\sum_{h \in \mathcal{L}} z_{i h} d_h\right)\right] x_{i j}}
\end{equation*}
and
\begin{equation*}
   w_j=u_{ 0j}+\sum_{i \in \mathcal{N}}\left[\hat{V}_{i j}-\beta_j \hat{P}_{i j}\left(\sum_{h \in \mathcal{L}} z_{i h} d_h\right)\right] x_{i j}.
\end{equation*}
Then, problem (\texttt{JDPEW}) can be transformed to (\texttt{JDPEW-1}), as follows:
\begin{align}
  \min \,\ & \sum_{j \in \mathcal{M}} \lambda_j u_{0 j} \hat{R}_j t_j+\sum_{i \in \mathcal{N}} \theta y_i 
  +\sum_{j \in \mathcal{M}} \lambda_j  \notag \\
  & \quad \times \sum_{i \in \mathcal{N}} \hat{V}_{i j} \Bigg[ \hat{R}_j x_{i j} - \hat{P}_{i j}\left(\sum_{h \in \mathcal{L}} z_{i h} x_{i j} d_h\right) +  \hat{C}_{i j}x_{ij}\Bigg] t_j \notag \\
  & \quad + \sum_{j \in \mathcal{M}} \lambda_j \sum_{i \in \mathcal{N}} \beta_j \hat{P}_{i j} \Bigg[- \hat{R}_j \left(\sum_{h \in \mathcal{L}} z_{i h} x_{i j} d_h\right) \notag \\
  & \quad + \hat{P}_{i j}\left(\sum_{h \in \mathcal{L}} z_{i h} x_{i j} d_h^2\right) - \left(\sum_{h \in \mathcal{L}} z_{i h} x_{i j} d_h\right) \hat{C}_{i j}\Bigg] t_j \\
  \text{s.t.} \,\ & \eqref{constraint1}-\eqref{constraint6} &&\notag \\
  & \,\ t_j w_j = 1 \,\ &&  \forall j \in \mathcal{M} \label{constraint:tw} \\
  & \,\ w_j = u_{ 0j} + \sum_{i \in \mathcal{N}}\left[\hat{V}_{i j} x_{i j} - \beta_j \hat{P}_{i j}\left(\sum_{h \in \mathcal{L}} z_{i h} x_{i j} d_h\right)\right] && \forall j \in\mathcal{M} \\
  & \,\ t_j \geq 0, \, w_j \geq 0 &&  \forall j \in \mathcal{M} \label{w}.
\end{align}
Let $s_{i j h}=z_{i h} x_{i j}$. According to \cite{mccormick1976computability}, we can replace constraint $s_{i j h} = z_{i h} x_{i j}$ with the following $4\times m \times n \times l$ constraints: 
\begin{align}
    s_{i j h} & \geq z_{i h}+\left(x_{i j}-1\right) &&   \forall  i \in \mathcal{N},~j \in \mathcal{M},~h \in \mathcal{L}\label{s_start}\\
    s_{i j h} & \geq 0 &&   \forall  i \in \mathcal{N},~j \in \mathcal{M},~h \in \mathcal{L} \\
    s_{i j h} & \leq z_{i h} &&   \forall  i \in \mathcal{N},~j \in \mathcal{M},~h \in \mathcal{L} \qquad\\
    s_{i j h} & \leq x_{i j} &&   \forall  i \in \mathcal{N},~j \in \mathcal{M},~h \in \mathcal{L}. \label{s_end}
\end{align}
Then, problem (\texttt{JDPEW-1}) can be further transformed to (\texttt{JDPEW-2}), shown below:
\begin{align}
  \min \,\ & \sum_{j \in \mathcal{M}} \lambda_j u_{0 j} \hat{R}_j t_j+\sum_{i \in \mathcal{N}} \theta y_i 
  +\sum_{j \in \mathcal{M}} \lambda_j \notag \\
  & \quad \times \sum_{i \in \mathcal{N}} \hat{V}_{i j} \Bigg[\hat{R}_j x_{i j} - \hat{P}_{i j}\left(\sum_{h \in \mathcal{L}} s_{ij h}  d_h\right) + \hat{C}_{i j}x_{ij}\Bigg] t_j \notag \\
  & \quad +\sum_{j \in \mathcal{M}} \lambda_j \sum_{i \in \mathcal{N}} \beta_j \hat{P}_{i j} \Bigg[- \hat{R}_j \left(\sum_{h \in \mathcal{L}} s_{ij h} d_h\right) \notag \\
  & \quad + \hat{P}_{i j}\left(\sum_{h \in \mathcal{L}} s_{ij h} d_h^2\right) - \left(\sum_{h \in \mathcal{L}} s_{ij h} d_h\right) \hat{C}_{i j}\Bigg] t_j \label{obj-22} \\
  \text{s.t.} \,\ & \eqref{constraint1}-\eqref{constraint6},\, \eqref{constraint:tw},\, \text{and} \, \eqref{w}-\eqref{s_end} & \notag \\
  & w_j=u_{ 0j} + \sum_{i \in \mathcal{N}}\left[\hat{V}_{i j} x_{i j}-\beta_j \hat{P}_{i j}\left(\sum_{h \in \mathcal{L}} s_{ij h}  d_h\right)\right] & \forall j \in \mathcal{M}. \label{wj}
\end{align}

Let $g_{i j}=t_j x_{i j}$ and $o_{i j h}=\mathrm{s}_{i j h} t_j$. Similarly, we can replace these two equalities with constraints \eqref{g_start}--\eqref{g_end} and \eqref{o_start}--\eqref{o_end}, respectively. Meanwhile, the objective function in \eqref{obj-22} can be rewritten as the objective function in \eqref{obj-26}.
Observing that $w_j \geq 0$, we can state constraint~\eqref{constraint:tw} in the rotated cone form:
\begin{equation}
    w_j t_j \geq 1 \quad  \forall  j \in \mathcal{M} \label{wjtj}.
\end{equation}
By doing so, we reformulate (\texttt{JDPEW}) as the following problem, denoted by (\texttt{JDPEW-MISOCP}):
\begin{align}
    \min \,\ & \sum_{j \in \mathcal{M}} \lambda_j u_{0 j} \hat{R}_j t_j+\sum_{i \in \mathcal{N}} \theta y_i 
    +\sum_{j \in \mathcal{M}} \lambda_j \notag\\
    & \quad \times \sum_{i \in \mathcal{N}} \hat{V}_{i j} \Bigg[\hat{R}_j g_{i j} - \hat{P}_{i j}\left(\sum_{h \in \mathcal{L}} o_{ij h}  d_h\right) + \hat{C}_{i j}g_{ij} \Bigg] \notag\\
    & \quad + \sum_{j \in \mathcal{M}} \lambda_j \sum_{i \in \mathcal{N}} \beta_j \hat{P}_{i j} \Bigg[- \hat{R}_j \left(\sum_{h \in \mathcal{L}} o_{ij h}  d_h\right) \notag\\ 
    & \quad + \hat{P}_{i j}\left(\sum_{h \in \mathcal{L}} o_{ij h} d_h^2\right) - \left(\sum_{h \in \mathcal{L}} o_{ij h} d_h\right) \hat{C}_{i j} \Bigg] \label{obj-26} \\
    \text{s.t.} \,\ & \eqref{constraint1}-\eqref{constraint6},\,
     \eqref{w}-\eqref{s_end}\, \text{and}\, \eqref{wj}-\eqref{wjtj} & \notag \\
    & g_{i j} \geq t_j+\left(x_{i j}-1\right) t_j^U \,\ && \forall i \in \mathcal{N},~j \in \mathcal{M} \label{g_start} \\ 
    & g_{i j} \geq x_{i j} t_j^L \,\ && \forall i \in \mathcal{N},~j \in \mathcal{M}  \\
    & g_{i j} \leq t_j+\left(x_{i j}-1\right) t_j^L \,\ &&  \forall i \in \mathcal{N},~j \in \mathcal{M}\\
    & g_{i j} \leq x_{i j} t_j^U \,\ &&  \forall i \in \mathcal{N},~j \in \mathcal{M}  \label{g_end} \\ 
    & o_{j i h} \geq s_{i j h} t_j^L \,\ &&  \forall i \in \mathcal{N},~j \in \mathcal{M},~h \in \mathcal{L} \label{o_start} \\ 
    & o_{j i h} \geq t_j+\left(s_{i j h}-1\right) t_j^U \,\ && \forall i \in \mathcal{N},~j \in \mathcal{M},~h \in \mathcal{L}\\
    & o_{i j h} \leq t_j+\left(s_{i j h}-1\right) t_j^L \,\ && \forall i \in \mathcal{N},~j \in \mathcal{M},~h \in \mathcal{L} \\
    & o_{i j h} \leq s_{i j h} t_j^U \,\ && \forall i \in \mathcal{N},~j \in \mathcal{M},~h \in \mathcal{L} \label{o_end}\\
    & y_i \in\{0,1\},\, x_{i j} \in\{0,1\},\, z_{i h} \in\{0,1\}, \, p_{ij}\geq 0 \,\ &&  \forall i \in \mathcal{N},~j \in \mathcal{M},~h \in \mathcal{L},
\end{align}
where $t_j^L = 1/(u_{ 0j}+\sum_{i \in \mathcal{N}} \hat{V}_{i j})$ and $t_j^U=1/u_{ 0j}$.

After the aforementioned transformation, the original problem (\texttt{JDPEW}) is simplified to an MISOCP problem, which can be solved by commercial solvers (e.g., Gurobi) directly. 
Indeed, the following proposition shows that the two optimization problems are equivalent.

\begin{prop}\label{prop2}
    The optimal solution $(\mathbf{X}^*, \mathbf{y}^{*}, \mathbf{Z}^{*}, \mathbf{P}^{*})$ obtained by solving (\texttt{JDPEW-MISOCP}) is also optimal for (\texttt{JDPEW}).
\end{prop}

\subsection{ITS Approach}\label{Subsec:twostage}
Although the MISOCP reformulation approach can obtain the optimal solution, it has high computational complexity, especially when the number of decision variables and constraints are large. Therefore, as the number of subsystems increases, the computational complexity of the MISOCP reformulation approach grows exponentially, requiring a substantial amount of time to reach the optimal solution. 

To improve computational efficiency, we propose an alternative approach that solves the joint optimization problem (\texttt{JDPEW}) by an ITS scheme. It consists of two steps: EW design and EW pricing. In the first step, subsystems are bundled and recommended in EW contracts with a fixed price discount. In the second step, the pricing discount is optimized based on the EW contracts obtained in the first step. The two steps are implemented iteratively until all decision variables are convergent. We detail the sub-optimization problems in the two steps as follows.

\textbf{Step 1: EW Design.} In this step, we focus on the problem of designing EW contracts with different subsystems for multiple groups of products from the EW provider's perspective. Given the initial prices of all subsystems and the pre-specified discounts, the optimal subsystem recommendation can be obtained by solving the following optimization problem: 
\begin{align*}
  \max \,\ & \sum_{j \in \mathcal{M}} \lambda_j  \sum_{i \in \mathcal{N}}q_{i j}\left(\mathbf{p}_{j}^{*},\mathbf{x}_{j}\right) \\
  & \quad \times \left(p^{*}_{ij} - \sum_{k \in \mathcal{W}} \mathbb{I}_{i k} f_{k j} c_{kj}\right) - \sum_{i \in \mathcal{N}} \theta y_i \\
  \text {s.t.}\,\  & x_{i j} \leq y_i &&  \forall i \in \mathcal{N},~j \in \mathcal{M} \\ 
  \,\ & \sum_{i \in \mathcal{N}} \mathbb{I}_{i k} x_{i j} \geq 1 && \forall k \in \mathcal{W},~j \in \mathcal{M} \\
  \,\ & y_i \in\{0,1\},\, x_{i j} \in\{0,1\} && \forall i \in \mathcal{N},~j \in \mathcal{M},
\end{align*}
where $q_{i j}(\mathbf{p}_{j}^{*},\mathbf{x}_{j})$ is the choice probability of group-$j$ customers for EW contract $i$, which is given by Equation~\eqref{choice prob}, with $\mathbf{p}^{*}_{j}$ replacing $\mathbf{p}_{j}$.

The objective is the same as that of the original problem (\texttt{JDPEW}), that is, to maximize the expected total profit for the EW provider. However, to reduce computational complexity, the decision variables in this step are only related to EW design (i.e., $\mathbf{X}$ and $\mathbf{y}$), provided that the EW price $\mathbf{P}^{*}$ and discount $\mathbf{Z}^{*}$ are given. Meanwhile, only constraints related to $\mathbf{X}$ and $\mathbf{y}$ are retained. 

\textbf{Step 2: EW Pricing.} In this step, the aim is to determine the prices of EW contracts designed in Step 1 by optimizing the discounts. The price optimization problem is formulated as follows:
\begin{align*}
  \max \,\ & \sum_{j \in \mathcal{M}} \lambda_j \sum_{i \in \mathcal{N}}q_{i j}\left(\mathbf{p}_{j}, \mathbf{x}_{j}^*\right) \\
  & \quad \times \left[p_{ij}-\sum_{k \in \mathcal{W}} \mathbb{I}_{i k} f_{k j} c_{kj}\right]-\sum_{i \in \mathcal{N}} \theta y_i^*\\
  \text{s.t.}\,\ & \sum_{h \in \mathcal{L}} z_{i h}=1 && \forall i \in \mathcal{N} \\
  & \,\ \left(\sum_{h \in \mathcal{L}} z_{i h} d_h-\sum_{h^{\prime} \in L} z_{i^{\prime} h^{\prime}} d_{h^{\prime}}\right) \\
  & \quad \times \left(\sum_{k \in \mathcal{W}} \mathbb{I}_{i k}-\sum_{k \in \mathcal{W}} \mathbb{I}_{i^{\prime} k}\right) \leq 0 &&    \forall i \in \mathcal{N},~ i^{\prime} \in N \\
  & \,\ p_{ij}=\left(\sum_{k\in \mathcal{W}}\mathbb{I}_{ik}P_{kj}^0\right)\left(\sum_{h\in \mathcal{L}}z_{ih}d_{h}\right) &&  \forall i \in \mathcal{N},~j \in \mathcal{M} \\
  & \,\ z_{ih} \in\{0,1\},\, p_{ij}\geq 0 &&  \forall i \in \mathcal{N},~h \in \mathcal{L}, 
\end{align*}
where $q_{i j}(\mathbf{p}_{j}, \mathbf{x}^{*}_{j})$ is the choice probability of group-$j$ customers for EW contract $i$, which is given by Equation~\eqref{choice prob}, with $\mathbf{x}_j^{*}$ replacing $\mathbf{x}_j$. Given $\mathbf{X}^*$ and $\mathbf{y}^*$ obtained from Step 1, we focus only on optimizing the bundle prices in Step 2, so the corresponding decision variables are $\mathbf{Z}$ and $\mathbf{P}$.

Although the problem scale is reduced by two-step separation, the sub-problems in Steps 1 and 2 still need equivalent transformation so that they can be solved by commercial solvers directly. Detailed description on the transformations in the two steps can be found in Supplementary Online Materials (Appendix~\ref{app:reform}).
However, sequentially solving the sub-problems in two steps does not guarantee optimality for the joint optimization problem. We employ block coordinate descent (BCD) techniques and iterate through the two steps until convergence is reached for both solutions. More details are presented in Algorithm~\ref{alg:algorithm1}.

We further discuss the theoretical properties of the BCD algorithm in solving the optimization problem (\texttt{JDPEW}). 
While the ITS approach provides an approximation for the original problem (\texttt{JDPEW}) and does not guarantee an optimal solution within a limited computation time, it still possesses several desirable theoretical properties related to the stability and optimality in each iteration.
Specifically, if we define the objective function of problem (\texttt{JDPEW}) as $g(\cdot)$, that is, $g(\mathbf{X},\mathbf{y},\mathbf{P},\mathbf{Z}):=\sum_{j \in \mathcal{M}} \lambda_j \sum_{i \in \mathcal{N}} q_{i j}\left(\mathbf{p}_{j}, \mathbf{x}_{j}\right) \cdot\left(p_{ij} - \sum_{k \in \mathcal{W}} \mathbb{I}_{i k} f_{k j} c_{kj}\right)-\sum_{i \in \mathcal{N}} \theta y_i$, then we have the following propositions.
\begin{prop}[Stability] \label{prop3}
 	Let $(\mathbf{X}^{[\alpha]}, \mathbf{y}^{[\alpha]},\mathbf{P}^{[\alpha]},\mathbf{Z}^{[\alpha]})$ denote the solution obtained by Algorithm~\ref{alg:algorithm1} in the $\alpha$th iteration. Then, the sequence of objective value $g(\mathbf{X}^{[\alpha]}, \mathbf{y}^{[\alpha]},\mathbf{P}^{[\alpha]},\mathbf{Z}^{[\alpha]})$ generated by the BCD-based algorithm will converge to a constant value.     
\end{prop}
\begin{prop}[Optimality within iteration]\label{prop4}
    After a finite number of iterations, the solutions to problem (\texttt{JDPEW}), obtained by Algorithm~\ref{alg:algorithm1}, have the following properties:
\begin{enumerate}[label=(\alph*)]
  \item For any $\mathbf{P}^{\prime}\in \mathbb{R}_+^{n\times m}$ and $\mathbf{Z}^{\prime} \in \{0,1\}^{n\times l}$, $g(\mathbf{X}^*, \mathbf{y}^*,\mathbf{P}^{\prime},\mathbf{Z}^{\prime}) \leq g(\mathbf{X}^*, \mathbf{y}^*,\mathbf{P}^*,\mathbf{Z}^*)$.
  \item For any $\mathbf{X}^{\prime} \in \{0,1\}^{n\times m} $ and $\mathbf{y}^{\prime} \in \{0,1\}^{n}$, $g(\mathbf{X}^{\prime}, \mathbf{y}^{\prime},\mathbf{P}^*,\mathbf{Z}^*) \leq g(\mathbf{X}^*, \mathbf{y}^*,\mathbf{P}^*,\mathbf{Z}^*)$.
\end{enumerate}
\end{prop}

\begin{algorithm}[t]
    \fontsize{8}{8}\selectfont
	\caption{BCD-based algorithm for (\texttt{JDPEW})}
	\label{alg:algorithm1}
	\KwIn{Coefficients $\theta$, $\{\lambda_j\}_{j\in \mathcal{M}}$, $\{\mathbb{I}_{ik}\}_{i \in \mathcal{N}, k \in \mathcal{W}}$, $\{P^{0}_{kj}\}_{k \in \mathcal{W}, j\in \mathcal{M}}$, $\{d_h\}_{h\in \mathcal{L}}$, $\{F_{kj}\}_{k\in \mathcal{W}, j\in \mathcal{M}}$, $\{c_{jk}\}_{j\in \mathcal{M},k\in \mathcal{W}}$}
	\KwOut{\ $\mathbf{X}^*$, $\mathbf{y}^*$, $\mathbf{P}^*$ and $\mathbf{Z}^*$}  
	\BlankLine
	Initialize $\mathbf{X}^{[0]}$, $\mathbf{y}^{[0]}$, $\mathbf{P}^{[0]}$, and $\mathbf{Z}^{[0]}$
    \BlankLine Set $ \mathbf{X}^* \leftarrow \mathbf{X}^{[0]}$,  $ \mathbf{y}^* \leftarrow \mathbf{y}^{[0]}$, $ \mathbf{P}^* \leftarrow \mathbf{P}^{[0]} $, $ \mathbf{Z}^* \leftarrow \mathbf{Z}^{[0]}$,
     $\alpha \leftarrow 1$	
	\BlankLine \While{\textnormal{$\mathbf{X}^*$, $\mathbf{y}^*$, $\mathbf{P}^*$ and $\mathbf{Z}^*$ do not converge}}{
		Obtain $\mathbf{X}^*$ and $\mathbf{y}^*$ by solving the EW design problem in (\textsc{Step-1}) based on $\mathbf{P}^{[\alpha-1]}$ and $\mathbf{Z}^{[\alpha-1]}$
        
        Set $\mathbf{X}^{[\alpha]} \leftarrow \mathbf{X}^*$ and $\mathbf{y}^{[\alpha]} \leftarrow \mathbf{y}^*$ 

        Obtain $\mathbf{Z}^*$ and $\mathbf{P}^*$ by solving the EW pricing problem in (\textsc{Step-2}) based on $\mathbf{X}^{[\alpha]}$ and $\mathbf{y}^{[\alpha]}$ 

        Set $\mathbf{Z}^{[\alpha]} \leftarrow \mathbf{Z}^*$, $\mathbf{P}^{[\alpha]} \leftarrow \mathbf{P}^*$, and $\alpha \leftarrow \alpha+1$
	}

\end{algorithm}

\subsection{Computation Complexity of the Two Approaches}
We are now in a position to discuss the computation complexities of the proposed solution approaches. We summarize in Table~\ref{tab:ComtwoAppro} the number of binary variables, continuous variables, linear constraints, and quadratic constraints during the solving processes based on the two approaches. In particular, for the ITS approach, we present the complexity for each step in a single iteration (note that this approach needs multiple iterations to converge). Recall that $m$ is the total number of products, $w$ is the number of subsystems, $n=2^w-1$ is the total number of potential EW contracts, and $l$ is the number of discount ladders. 
As $n$ is much larger than $m$, $w$, and $l$, the ITS approach can significantly reduce computational complexity by decreasing the number of decision variables. 
That is, although the ITS approach cannot guarantee global optimality, it still has potential advantage in computation, especially when the number of subsystems is large.

\begin{table}[t]
\centering
\fontsize{8}{10}\selectfont
\caption{Computational complexities of the two approaches}
\label{tab:ComtwoAppro}

\begin{subtable}[t]{1.0\textwidth}
\centering
\caption{Number of variables}
\label{tab:1a}
{\fontsize{8}{10}\selectfont 
\renewcommand{\arraystretch}{1.5} 
\begin{tabular}{>{\centering\arraybackslash}m{2.0cm}  
    >{\centering\arraybackslash}m{3.5cm}  
    >{\centering\arraybackslash}m{0.1cm}  
    >{\centering\arraybackslash}m{3.2cm}  
    >{\centering\arraybackslash}m{3.2cm} }
    \toprule
    \multirow{2}[4]{*}{ Variables} &  { MISOCP} &       & \multicolumn{2}{c}{\makecell{ITS (for each iteration)}} \\
    \cmidrule{2-2}\cmidrule{4-5}          & { Total} &       & {Step 1} & \multicolumn{1}{c}{Step 2}  \\
    \midrule
    {Binary} & $mn + nl + n$ &       & $mn + n$ & $nl$ \\
    {Continuous} &  $2mnl + 2mn + 2m$ &       & $mn + 2m$  & $mnl + mn +2m$  \\
    \bottomrule
    \end{tabular}
    }
\end{subtable}

\vspace{1em}

\begin{subtable}[t]{1.0\textwidth}
\centering

\caption{Number of constraints}
\label{tab:1b}
{\fontsize{8}{10}\selectfont
\renewcommand{\arraystretch}{1.5} 
\begin{tabular}{>{\centering\arraybackslash}m{2.0cm}  
    >{\centering\arraybackslash}m{3.5cm}  
    >{\centering\arraybackslash}m{0.1cm}  
    >{\centering\arraybackslash}m{3.2cm}  
    >{\centering\arraybackslash}m{3.2cm} } 
    \toprule
    \multirow{2}[4]{*}{ Variables} &  MISOCP &       & \multicolumn{2}{c}{ITS (for each iteration)} \\
    \cmidrule{2-2}\cmidrule{4-5}          & { Total} &       & { Step 1} & \multicolumn{1}{c}{ Step 2} \\
    \midrule
    {Linear} & $n^2+7mnl+6mn+2m+wm$ &       & $5mn+wm+m$ & $n^2+4mnl+mn+n+m$ \\ 
    { Quadratic} & $m$     &       & $m$     & $m$  \\
    \bottomrule
    \end{tabular}%
    }
\end{subtable}
\end{table}

\section{Numerical Studies} \label{Sec:Numerial}
In this section, we take third-party automobile EWs as a motivating example and conduct numerical experiments to verify the performance of the proposed solution approaches (i.e., MISOCP in Section~\ref{Subsec:socp} and ITS in Section~\ref{Subsec:twostage}).
We first give an overview of the experiment setting (Section \ref{Subsec:overview}), and then compare the performance of the proposed two approaches in solution optimality and computational efficiency (Section~\ref{Subsec:2appro}). We further compare the joint optimization model with three benchmarks in Section~\ref{Subsec:compbench}. Finally, we perform sensitivity analyses in Section~\ref{Subsec:sensitive} to analyze the effects of some important factors and draw managerial insights.

\subsection{Experiment Setting} \label{Subsec:overview}
We first provide details about the data and the parameter setting. The numerical results in this section are based on randomly generated data sets. If not stated otherwise, the parameters are set as follows. 

For parameters related to customers, we consider $m=5$ different groups, with the price of their automobiles increasing from group $j=1$ to $j=5$. We assume $\lambda_j=0.2$ for $j\in \{1,2,\dots,5\}$, which will be relaxed in Section~\ref{Subsec:sensitive}.  
According to \cite{chen2009consumers}, customers who own higher-value automobiles tend to be higher-value customers with a lower sensitivity to EW price. Therefore, we assume that the sensitivity parameter $\beta_j$ decreases with $j$. At the same time, consumers are more inclined to purchase EW services for higher-value automobiles \citep{padmanabhan1993warranty}. As a result, the preference weight of outside option $u_{0j}$ also decreases with $j$. 
Without loss of generality, we rank the subsystems in ascending order of failure cost. We further assume that the attraction value $v_{kj}$ for group-$j$ customers follows a uniform distribution.
Table~\ref{tab:Parameters setting} summarizes the aforementioned parameter setting.

\begin{table}[htbp]
  \centering
  \fontsize{8}{10}\selectfont
  \renewcommand{\arraystretch}{1.5} 
  \caption{Parameter setting}
    \begin{tabularx}{\textwidth}{>{\centering\arraybackslash}X>{\centering\arraybackslash}X>{\centering\arraybackslash}X>{\centering\arraybackslash}X>{\centering\arraybackslash}X >{\centering\arraybackslash}X>{\centering\arraybackslash}X>{\centering\arraybackslash}X>{\centering\arraybackslash}X}
    \toprule
    \multirow{2}[4]{*}{$j$} & \multirow{2}[4]{*}{$v_{kj}$} & \multirow{2}[4]{*}{$u_{0j}$} & \multirow{2}[4]{*}{$\beta_j$} & \multicolumn{5}{c}{ $c^{0}_{kj}$ ($\times 10^{3})$} \\
\cmidrule{5-9}          &       &       &       & $k=1$ & $k=2$ & $k=3$ & $k=4$ & $k=5$ \\
    \midrule
    1     & $U(20, 25)$ & 300   & 0.05  & 0.6   & 1.2 & 1.8  & 3.0  & 4.8 \\
    2     & $U(30, 35)$ & 250   & 0.04  & 3.0  & 3.6  & 4.2  & 5.4  & 6.0 \\
    3     & $U(35, 40)$ & 200   & 0.02  & 6.0  & 7.2  & 8.4  & 9.6  & 12.0 \\
    4     & $U(40, 45)$ & 100   & 0.005 & 12.0 & 15.0 & 18.0 & 21.0 & 30.0 \\
    5     & $U(45, 50)$ & 50    & 0.0001 & 30.0 & 36..0 & 42.0 & 48.0 & 54.0 \\
    \bottomrule
    \end{tabularx}%
  \label{tab:Parameters setting}%
\end{table}%

For parameters related to EW providers, the failure cost $c_{kj}$ of each subsystem $k$ is also presented in Table~\ref{tab:Parameters setting}. We assume that the initial price $P_{kj}^0$ before optimization is proportional to $c_{kj}$; that is, $c_{kj} = \gamma P_{kj}^0$ with $\gamma=6$. Generally, the parameter $\gamma$ is larger than 1; otherwise, customers would prefer to repair the failed automobiles rather than purchasing EWs. The advertising cost for promoting each EW contract is $\theta$. 
Finally, we assume that the failure probability of each subsystem follows a uniform distribution; that is, $F_{kj} \sim U(0.05,\,0.2)$.

We note that for all numerical experiments, the number of EW contracts, $n$, and the number of customer groups, $m$, are assumed to be identical. All numerical tests were conducted on a Windows 10 64-bit Intel Core i7-5600 with 2.0 gigahertz and 16 gigabytes of memory. The tests are implemented in Python 3.8 to use the Gurobi solver.

\subsection{Comparison of the Two Solution Approaches}\label{Subsec:2appro}
We now investigate the optimality and computational efficiency of the MISOCP reformulation and ITS solution approaches. 
In the MISOCP reformulation approach, we propose to transform (\texttt{JDPEW}) to (\texttt{JDPEW-MISOCP}), which can be solved by state-of-the-art commercial solvers.
As discussed previously, although this approach can achieve the optimal solution, it suffers from high computational complexity with significantly increased number of integer decision variables. 
In the ITS approach, however, we decompose the original problem (\texttt{JDPEW}) into two sequential sub-problems and solve them in an iterative manner. 
Specifically, we first set $\mathbf{Z}^0$ and $\mathbf{P}^0$ as the initial values of $\mathbf{Z}^{*}$ and $\mathbf{P}^*$, respectively. Specifically, the element in $\mathbf{Z}^0$ is calculated as 
\begin{equation*}
  z_{i h}^0 = 
  \begin{cases}
  1 & \text { if } h=\arg\max\left\{k\in \mathcal{W}: \mathbb{I}_{i k}=1\right\}, \\ 
  0 & \text { otherwise,}
  \end{cases}
\end{equation*}
for all $i\in \mathcal{N}$ and $h\in \mathcal{L}$, and the corresponding price is set to $p_{ij}^0=(\sum_{k\in \mathcal{W}}\mathbb{I}_{ik}P_{kj}^0)(\sum_{h\in \mathcal{L}}z_{ih}^0 d_{h})$. Then, we iterate the EW design and pricing process based on Algorithm~\ref{alg:algorithm1}. 

\begin{table}[t]

       \fontsize{8}{10}\selectfont
      \centering
      \caption{Profits and  computational performance of the proposed two approaches}
      \renewcommand{\arraystretch}{1.5} 
        \begin{tabularx}{\textwidth}{p{0.3cm}>{\centering\arraybackslash}X >{\centering\arraybackslash}X>{\centering\arraybackslash}X p{0.05cm}>{\centering\arraybackslash}X >{\centering\arraybackslash}X}
        \toprule
        \multirow{2}[4]{*}{$w$} & \multicolumn{3}{c}{ Profit} &       & \multicolumn{2}{c}{  Time (s)}\\
    \cmidrule{2-4}\cmidrule{6-7}          & {MISOCP} & { ITS} & {Gap (in \%)} &       & {MISOCP} & {ITS} \\
        \midrule
        3     & 2106.34  & 2106.33  & 0.0379 &       & 1437.24  & 1.72  \\
        4     & 3405.30  & 3405.35  & -0.0015 &       & 3600.00  & 195.31  \\
        5     & 4905.16  & 4908.67  & -0.0715 &       & 3600.00  & 693.35  \\
        \bottomrule
        \parbox{\textwidth}{
        \fontsize{8}{8}\selectfont
        \vspace{0.2cm}
        \textbf{Note:} Gap = [Profit(MISOCP)$-$Profit(ITS)]/Profit(ITS) $\times$ 100\%.
        }
        \end{tabularx}%
      \label{tab:comp-twoproposed}%
\end{table}%

Table~\ref{tab:comp-twoproposed} presents the average profits obtained by the two solution approaches over 30 replications.
It should be noted that we impose a limit of 3600 seconds for the MISOCP approach and set a maximum runtime of 300 seconds for each step of the ITS approach. 
As shown in Table~\ref{tab:comp-twoproposed}, as we restrict the computation time to 3600 seconds, it is insufficient for the MISOCP reformulation approach to reach optimal solutions; its advantage in optimality diminishes under limited computational time, especially when the number of subsystems becomes large. 
In contrast, as the number of subsystems increases (say, $w$ is more than 3), the ITS approach could achieve better solutions than its counterpart (in terms of average profits) within shorter computation time, making it particularly suitable for scenarios involving a large number of subsystems.

In addition, Figure~\ref{fig:comparison} illustrates the distributions of profit and computation time across the 30 replications. The result indicates that relative to the MISOCP reformulation approach, the ITS approach is able to attain comparable profit levels while requiring substantially less computation time. Importantly, this computational advantage becomes more pronounced as the number of subsystems increases.

In summary, if the scale of the optimization problem is small (say, $w$ is less than 4), the MISOCP reformulation approach is preferred; otherwise, the ITS approach performs significantly better for large-scale optimization problems.

\begin{figure}[t]
  \centering
  \begin{subfigure}[t]{0.8\textwidth}
    \centering
    \includegraphics[width=\textwidth]{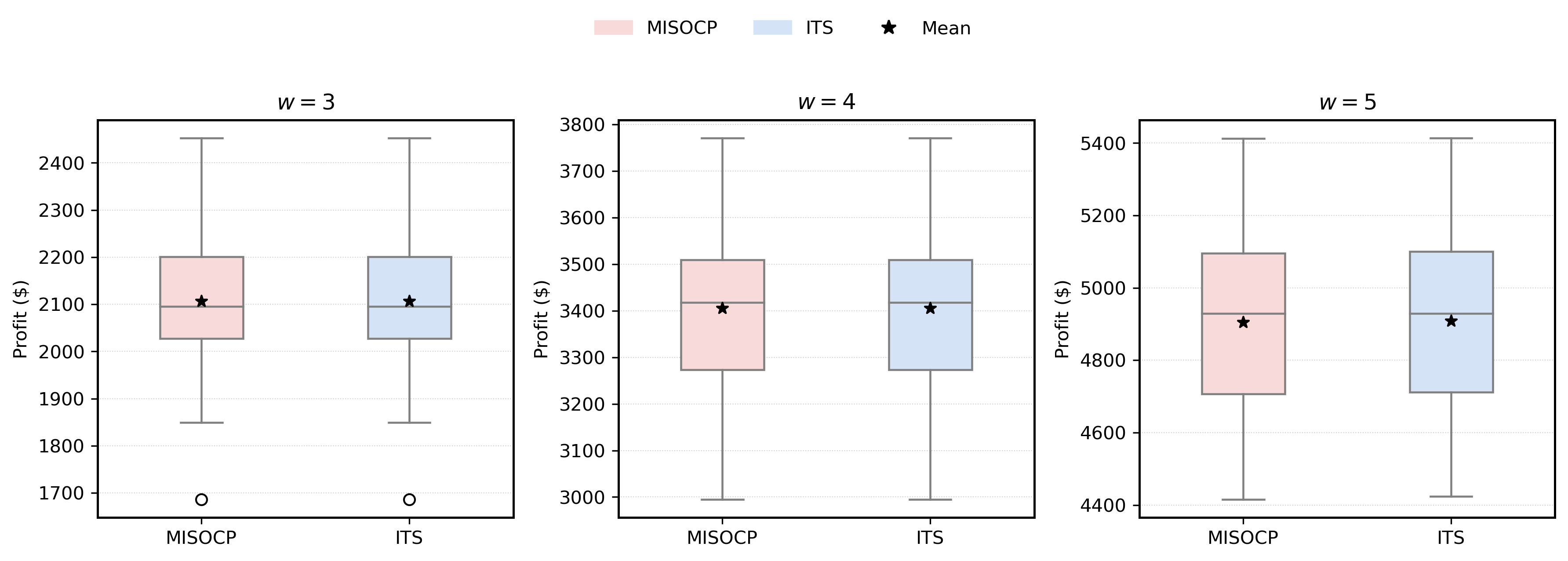} 
    \caption{Total profit}
    \label{fig:profit}
  \end{subfigure}
  
  \vspace{0.5cm} 
  
  \begin{subfigure}[t]{0.8\textwidth}
    \centering
    \includegraphics[width=\textwidth]{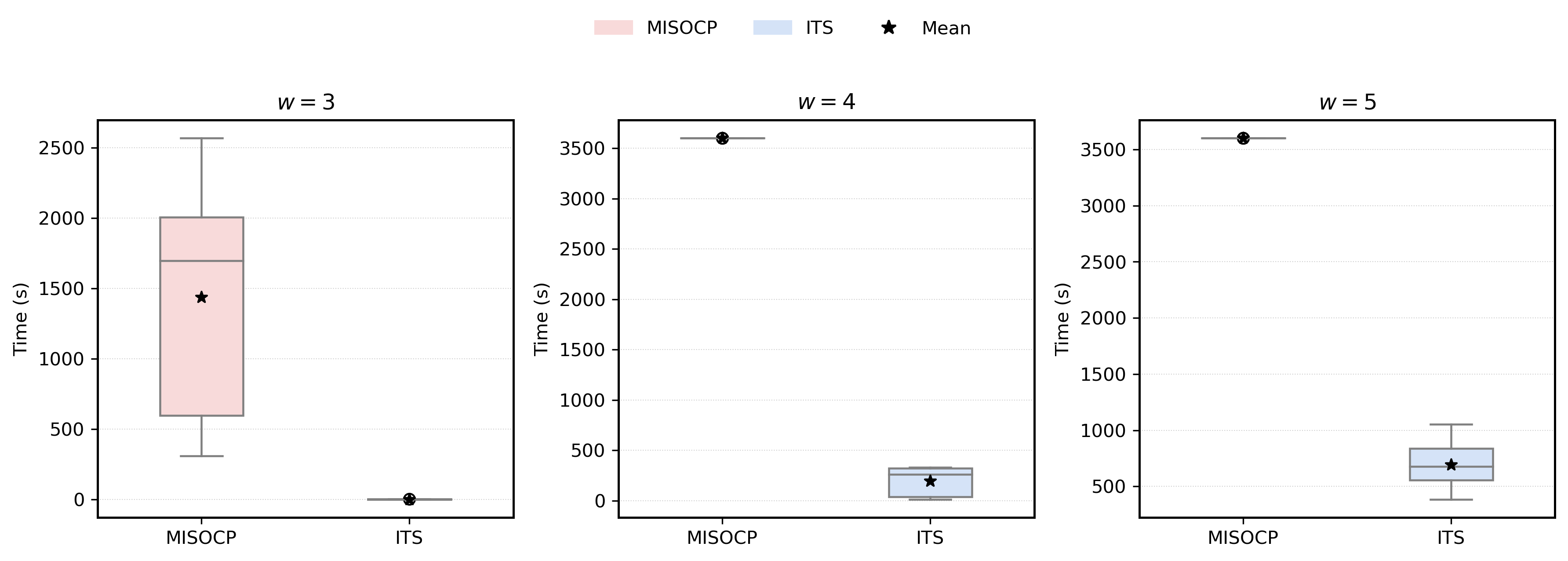} 
    \caption{Computation time}
    \label{fig:time}
  \end{subfigure}

  \caption{Performance comparison of the two proposed approaches}
  \label{fig:comparison}
\end{figure}

\subsection{Comparison with Other Benchmarks}\label{Subsec:compbench}
As our model is motivated by real-world challenges faced by third-party EW providers in the aftermarket, we introduce in Section~\ref{Sec:Bench} three benchmark strategies that are currently adopted in practice, and present the comparison results in Section~\ref{Sec:comp} to validate the effectiveness of the proposed joint optimization model.

\subsubsection{Benchmarks} \label{Sec:Bench}

To evaluate the necessity of jointly optimizing EW design and pricing in our proposed model, and motivated by the prevailing industry practice, we develop three benchmark strategies for comparison purposes. Benchmark 1 (\texttt{BM-1}) focuses on consistent EW design and ignores the heterogeneity of automobiles in different groups. Compared to \texttt{BM-1}, our proposed method has two main improvements: (1) personalized EW design and (2) joint optimization of EW design and pricing. To isolate and evaluate the contribution of each enhancement, we introduce two additional benchmarks: \texttt{BM-2} (personalized EW design) and \texttt{BM-3} (bundle pricing with consistent design), aiming to demonstrate the profit gain attributed to each enhancement. The benchmarks are detailed as follows:

\texttt{BM-1}: Consistent EW design. In this benchmark, a consistent EW contract with the same subsystem composition is designed for all groups of automobiles while ignoring the heterogeneity in automobile groups. 

\texttt{BM-2}: Personalized EW design. Different from \texttt{BM-1}, the diversities in different automobile groups are taken into account, which aligns with the formulation in Step 1 of the ITS approach as well as the current strategy adopted by Lizhen Company. This benchmark serves to assess the importance of pricing and the necessity of joint optimization.

\texttt{BM-3}: Bundle pricing with consistent design. This strategy jointly optimizes EW design and pricing but ignores the heterogeneity in different automobile groups. Compared to our proposed model, adopting a consistent EW design can reduce the complexity of problem-solving.  

More details on the optimization models and the corresponding solutions for the three benchmarks can be found in Supplementary Online Materials (Appendix~\ref{app:benchmark}).

\subsubsection{Performance Comparison} \label{Sec:comp}
To demonstrate the importance of joint EW design and pricing, we now compare our joint optimization model with the three benchmarks.
The performance of the four models is evaluated by the average profit over 30 replications. 
To demonstrate the superiority of the joint model, we examine the profit improvement of the joint model relative to each benchmark, denoted by Increment. 
Specifically, we define $\text{Increment} = \text{Profit}(\texttt{Joint}) - \text{Profit}(\texttt{BM})$ for each benchmark. 
We consider different initial EW prices for each subsystem by changing the ratio $\gamma$ (recall that $c_{kj} = \gamma P_{kj}^0$).
Table~\ref{tab:comp-3benchmark} presents the average profits and profit improvements under different combinations of $w$ and $\gamma$. It is worth noting that we employ both MISOCP and ITS approaches to solve the joint optimization model (the computation time is limited to 600 seconds) and report the higher profit value produced by the two approaches.


\begin{table}[ht!]
  \centering
  \fontsize{8}{10}\selectfont
  \caption{Performance comparison with three benchmarks}
 { 
\renewcommand{\arraystretch}{1.5} 
\begin{tabularx}{\textwidth}{>{\centering\arraybackslash}X p{1cm} >{\centering\arraybackslash}X>{\centering\arraybackslash}X>{\centering\arraybackslash}X>{\centering\arraybackslash}X p{0.5cm} >{\centering\arraybackslash}X>{\centering\arraybackslash}X>{\centering\arraybackslash}X}
    \toprule
    \multirow{2}[4]{*}{$w$} & \multirow{2}[4]{*}{$\gamma$} & \multicolumn{4}{c}{Profit} & \multirow{2}[4]{*}{} & \multicolumn{3}{c}{Increment} \\
\cmidrule{3-6}\cmidrule{8-10}          &       & \texttt{Joint} & \texttt{BM-1} & \texttt{BM-2} & \texttt{BM-3} &       & \texttt{BM-1} & \texttt{BM-2} & \texttt{BM-3} \\
    \midrule
    3 & 6     & 2106.34 & 1772.17 &  1802.06 & 2039.82 &       & 334.17  &  304.28 & 66.52  \\
          & 8     & 1180.28 & 934.81 &  948.61 & 1152.24 &       & 245.47  & 231.67 & 28.04  \\
          & 10     & 606.19 & 415.56  &  425.17  & 594.84 &       & 190.63 &  181.02  & 11.35  \\
    4 & 6    & 3404.39 & 2647.01 & 2685.21 & 3304.56 &       &  757.38  &  719.18 
     & 99.83 \\
          & 8     & 1927.44 & 1334.06 &  1355.63 & 1876.41 &       &  593.38 &   571.81 & 51.03 \\
          & 10     & 993.44 
    & 553.15 &  585.24 
  & 967.08 
 &       &  440.29 
 &  408.20 
   & 26.37 
  \\
    5 & 6      & 4900.85 & 3473.81 &  3529.53 & 4698.93  &       &  1427.04 &  1371.32 & 201.93   \\
          & 8      & 2807.42 & 1676.89 &  1716.60 & 2721.39 &    &  1130.53 
  & 1090.82 
 & 86.03 
 \\
          & 10     & 1433.23 
  & 688.34 &  730.77 & 1417.12 &       &  744.90 & 702.47 & 32.43  \\
    \bottomrule

    \end{tabularx}}%
  \label{tab:comp-3benchmark}%
\end{table}%

We can see from Table~\ref{tab:comp-3benchmark} that for all parameter combinations considered, the joint model always leads to higher total profits than the three benchmarks, demonstrating its superiority. 
In particular, as the number of subsystems $w$ increases, the advantage of joint optimization tends to become more prominent, reflected by the increase in Increment over $w$. 
In addition, we also observe that as $\gamma$ increases, the total profits for all models become lower. This is because for a given failure cost $c_{kj}$ of each subsystem, a larger $\gamma$ corresponds to a lower initial price $P_{kj}^0$, resulting in decreased total profit.
Furthermore, we find that for a fixed number of subsystems, as $\gamma$ increases (i.e., the initial price becomes lower), the Increment exhibits a declining trend. 
This indicates that the advantage of the proposed model is more significant when the initial price is higher, because such condition facilitates our joint model to maximize profit through optimal discount in pricing.

\subsection{Sensitivity Analyses} \label{Subsec:sensitive}

We now conduct sensitivity analyses on the composition of different customer groups (Section~\ref{Subsubsec:comcus}) and the subsystems' failure probabilities (Section~\ref{Subsubsec:failure}). 
By doing so, we attempt to derive managerial insights to guide the EW design and pricing practice. 

\subsubsection{Proportions of Customer Groups}\label{Subsubsec:comcus} 

Customers who purchase automobiles from different groups might have varying valuations for EW contracts, and the proportions of different customer groups in the market might affect the results. To assess this effect, we consider five groups of customers (i.e., $m=5$) with increasing values: the first group has the lowest value, while the fifth group has the highest. 
We examine the following three cases of customer proportions. Case 1: The five customer groups have uniform proportions (i.e., $\lambda_j=0.2$, $j\in\{1,2,\dots,5$\}), which is the same as the setting in Section~\ref{Subsec:overview}. Case 2: The proportion of customers in each group shows a decreasing trend (specifically, $\lambda_1=0.4$, $\lambda_2=0.2$, $\lambda_3=0.15$, $\lambda_4=0.1$, $\lambda_5=0.05$). Case 3: The customer distribution is symmetric, with a high proportion for the middle group (specifically, $\lambda_1=0.1$, $\lambda_2=0.2$, $\lambda_3=0.4$, $\lambda_4=0.2$, $\lambda_5=0.1$). Specifically, we define Benefit (in \%) to evaluate the advantages of our proposed model; specifically, $\text{Benefit} = [\text{Profit}(\texttt{Joint}) - \text{Profit}(\texttt{BM})]/ \text{Profit(\texttt{BM})}\times 100\% $.

\begin{table}[t]
  \centering
  \fontsize{8}{10}\selectfont 
  \caption{Effect of customer composition on total profits}
 {
\renewcommand{\arraystretch}{1.5} 
\begin{tabularx}{\textwidth}{>{\centering\arraybackslash}X p{1.0cm} >{\centering\arraybackslash}X>{\centering\arraybackslash}X>{\centering\arraybackslash}X>{\centering\arraybackslash}X p{0.5cm} >{\centering\arraybackslash}X>{\centering\arraybackslash}X>{\centering\arraybackslash}X}
    \toprule
    \multirow{2}[4]{*}{$w$} & \multirow{2}[4]{*}{Cases} & \multicolumn{4}{c}{Profit}    &       & \multicolumn{3}{c}{Benefit (in \%)} \\
\cmidrule{3-6}\cmidrule{8-10}          &       & \texttt{Joint} & \texttt{BM-1}  & \texttt{BM-2}  & \texttt{BM-3}  &       & \texttt{BM-1}  & \texttt{BM-2}  & \texttt{BM-3} \\
    \midrule
    3     & Case 1 & 2076.29 & 1742.9 & 1771.28 & 2009.86 &       & 19.29 & 17.42 & 3.33 \\
          & Case 2 & 730.83 & 621.96 & 634.01 & 705.64 &       & 17.56 & 15.33 & 3.57 \\
          & Case 3 & 1470.21 & 1253.31 & 1277.25 & 1417.63 &       & 17.36 & 15.18 & 3.70 \\
    4     & Case 1 & 3381.75 & 2621.66 & 2664.64 & 3284.39 &       & 29.09 & 27.03 & 2.97 \\
          & Case 2 & 1240.38 & 997.62 & 1010.04 & 1211.55 &       & 24.39 & 22.87 & 2.37 \\
          & Case 3 & 2472.64 & 1966.35 & 1997.76 & 2404.73 &       & 25.86 & 23.90  & 2.82 \\
    5     & Case 1 & 4926.63 & 3516.56 & 3574.34 & 4749.31 &       & 40.25 & 38.02 & 3.80 \\
          & Case 2 & 1820.51 & 1372.31 & 1393.99 & 1708.69 &       & 32.74 & 30.69 & 6.65 \\
          & Case 3 & 3584.12 & 2625.48 & 2672.33 & 3414.00  &       & 36.68 & 34.31 & 5.04 \\
    \bottomrule
    
    \end{tabularx}}%
  \label{tab:customer-comp}%
\end{table}%

Table~\ref{tab:customer-comp} presents the total profits and the Benefit values under different customer proportion settings, which further demonstrates the advantages of joint optimization.
More importantly, we find that Case 1 leads to the highest profit regardless of the number of subsystems, $w$. Even in Case 3 where medium-value customers dominate the market, the total profit is still lower than that in Case 1 with uniform proportions. These observations highlight the importance of high-value customers, who serve as the primary profit source for EW providers. 
In addition, we find that the profit improvements of the joint design and pricing model relative to the benchmarks (particularly, \texttt{BM-1} and \texttt{BM-2}) in Case 1 are fairly larger than those in Cases 2 and 3. 
Because more high-value customers in Case 1 typically generate higher marginal returns, they are charged higher EW prices while exhibiting relatively lower price sensitivity. Therefore, attracting and retaining high-value customers is a critical consideration in EW promotion strategies.

\subsubsection{Effect of Failure Probability}\label{Subsubsec:failure}

Failure probabilities of the subsystems are critical factors that providers must consider during EW design and pricing. We now evaluate the impact of subsystem failure probabilities on the results. To this end, we consider the settings where failure probabilities of subsystems follow homogeneous and heterogeneous distributions.

In the homogeneous, uniform-distribution setting, failure probability of each subsystem is assumed to follow $U(0.01,0.05)$ in the low failure-probability case (denoted by HU-L) and $U(0.05,0.10)$ in the high failure-probability case (denoted by HU-H). In the heterogeneous, uniform-distribution setting (denoted by HeU), failure probabilities are generated as $F_{1j} \sim U(0, 0.03)$, $F_{2j} \sim U(0.02, 0.05)$, $F_{3j} \sim U(0.03, 0.06)$, $F_{4j} \sim U(0.04, 0.07)$, and  $F_{5j} \sim U(0.05, 0.08)$ for $j \in \{1,2,\dots,5\}$. In the setting of mixed uniform and Normal distributions (denoted by UN-M), we set $F_{1j} \sim U(0, 0.05)$, $F_{3j} \sim U(0.06, 0.11)$, $F_{5j} \sim U(0.04, 0.09)$ and $F_{2j} \sim N(0.05, 0.1^2)$, $F_{4j} \sim N(0.04, 0.1^2)$, both truncated to $[0, 1]$.

Table~\ref{tab10} reports the total profits and the corresponding Benefit values in different failure-probability settings. By comparing the HU-L and HU-H scenarios, we find that a higher failure probability significantly decreases the EW provider's profit. 
Interestingly, by comparing the proposed joint optimization model with \texttt{BM-1} and \texttt{BM-2}, we find that the advantage of the joint model becomes significant especially when the failure probability is high or exhibits substantial heterogeneity. 
This is because the flexible pricing strategies allow EW providers to implement price discrimination based on failure risk, charging higher prices for high-risk products and lower prices for low-risk ones, thereby achieving higher profits.
Compared to \texttt{BM-3}, which is a simplified joint optimization model that recommends a uniform EW design across different product tiers to reduce computational complexity, Table \ref{tab10} shows that the profit gap between \texttt{BM-3} and the joint model narrows as the failure probability increases, but widens with the number of subsystems. This finding suggests that in scenarios with a limited number of subsystems or a higher failure probability, \texttt{BM-3} can be considered as a viable substitute for the joint optimization model.

We also examine scenarios in which there are correlations in the failure probabilities of different subsystems to further validate the optimality and robustness of our proposed method. More details are provided in Appendix~\ref{Sec:Rob}.

\begin{table}[t]
  \centering
  \fontsize{8}{10}\selectfont 
  \caption{Effect of failure probability on total profits}
     {
    \renewcommand{\arraystretch}{1.5} 
    \begin{tabularx}{\textwidth}{>{\centering\arraybackslash}X p{2cm} >{\centering\arraybackslash}X>{\centering\arraybackslash}X>{\centering\arraybackslash}X>{\centering\arraybackslash}X p{0.1cm} >{\centering\arraybackslash}X>{\centering\arraybackslash}X>{\centering\arraybackslash}X}
    \toprule
   \multirow{2}[4]{*}{$w$} & \multirow{2}[4]{*}{\makecell[l]{Failure\\ Probability}} & \multicolumn{4}{c}{{ Profit}} & \multirow{2}[1]{*}{} & \multicolumn{3}{c}{{ Benefit (in \%)}} \\
    \cmidrule{3-6}\cmidrule{8-10} &  & \texttt{Joint} & \texttt{BM-1} & \texttt{BM-2} & \texttt{BM-3} &  & \texttt{BM-1} & \texttt{BM-2} & \texttt{BM-3} \\
    \midrule
    3 & HU-L & 3166.79 & 2793.49 & 2857.44 &  3053.02  &   & 13.40   & 10.86  &  3.74  \\
    & HU-H &2078.83 & 1742.91 &  1775.35 &  2013.81  &   &  19.37   &  17.20  & 3.24 \\
    & HeU & 3047.45  & 2655.99 & 2719.69 & 2898.72 &       & 14.74  & 12.05 & 5.13  \\
     & UN-M & 2486.49 &  2132.21  & 2187.66 & 2400.82 &   & {16.62}  & {13.66}  & {3.57}  \\
    4 & HU-L & {5059.85} & {4277.58} & { 4361.40} & {4897.96} &  & { 18.30}   & {16.03}  & {3.31} \\
   & HU-H & {3397.43} & {2636.53} & { 2676.67} & {3294.78} &  & { 29.00}   & {27.08}  & { 3.13} \\
    & HeU & {4637.83} & {3870.54} & {3900.61} & {4490.06} &       & {19.82}  &  {18.90} & {3.29}  \\
    & UN-M & {4163.90} & {3353.49} & {3496.11} & {3994.30} &  & {24.17}  & { 19.10}  & { 4.25}  \\
   5 & HU-L & {7278.56} & {5865.17} & {5975.62} & {6942.29} &       & {24.12}  & { 21.83}  & {4.88}  \\
    & HU-H & {4919.05} & {3512.21} & {3565.57} & {4731.95} &       & {40.12}  & { 38.05}  & {4.00}  \\
    & HeU & {6361.35} & {4969.19} & {4999.07} & {6096.25} &       & {28.02}  & {27.25}  & {4.35}  \\
    & UN-M & {5865.41} & {4407.43} & {4577.36} & {5555.88} &       & {33.08}  &  {28.14} & {5.57}  \\
    \bottomrule

    \end{tabularx}}
     \label{tab10}%
\end{table}%

\section{Conclusions}\label{Sec:con}
 
This work addresses a prevalent optimization problem faced by third-party EW providers in the aftermarket. 
We develop a joint EW design and pricing model for multi-tier products serviced by a third-party EW provider, with the aim of maximizing the expected total profit. 
Given the computational complexity of the model, we propose two solution approaches, namely, the MISOCP formulation and ITS approaches.
We develop corresponding algorithms for each approach to provide detailed procedural descriptions, and discuss the corresponding theoretical properties regarding stability and optimality of the solutions.
Numerical experiments demonstrate that the joint optimization model significantly increases the provider's expected total profit and exhibits strong robustness. 
Several managerial insights are also identified: First, customers who own higher-priced products contribute more significantly to the total profit, therefore EW providers should place greater emphasis on maintaining a high proportion of such customers and ensuring their satisfaction. Second, when the failure probabilities become higher and/or the number of subsystems gets larger, the advantages of jointly optimizing EW design and pricing become more pronounced, as simultaneously addressing both aspects allows for greater flexibility in mitigating the risks associated with high failure probabilities and system complexity.

Our paper is the pioneer work to study the design and pricing of EW contracts with differentiated compositions and prices. Along with our work, there are several potential directions for future research. 
First, as the Internet of Things becomes increasingly prevalent, it is essential to incorporate data on customers’ personalized usage behaviors into the design and pricing of EW contracts.
Second, our current model may not fully capture the impact of the number of covered subsystems and their potential interactions. Future research could extend the utility specification to incorporate subsystem interactions, enabling a more realistic representation of customer preferences.

\section*{Disclosure Statement}
No potential conflict of interest was reported by the authors. 

\section*{Data Availability Statement}
The authors confirm that the data supporting the findings of this study are available within the article.



\bibliographystyle{chicago}
\spacingset{1}
\bibliography{IISE-Trans}

\newpage
\setcounter{page}{1}

\appendix

\begin{spacing}{1.5}
\begin{center}
\textbf{\Large Supplemental Online Materials to ``Joint Bundle Design and Pricing for Extended Warranty Providers Servicing Multi-Tier Products"}
\end{center}

\vspace{-2em}

\section{Reformulations of Problem (\texttt{JDPEW}) Based on the ITS Approach} \label{app:reform}

\subsection{Reformulation of Step 1}
Let 
\begin{equation}
   t_j=\frac{1}{u_{ 0j}+\sum_{i \in \mathcal{N}}\left[\hat{V}_{i j}-\beta_j \hat{P}_{i j}\left(\sum_{h \in \mathcal{L}} z_{i h}^{*} d_h\right)\right] x_{i j}}\notag
\end{equation}
and 
\begin{equation}
   w_j=u_{ 0j}+\sum_{i \in \mathcal{N}}\left[\hat{V}_{i j}-\beta_j \hat{P}_{i j}\left(\sum_{h \in \mathcal{L}} z_{i h}^{*} d_h\right)\right] x_{i j}.\notag
\end{equation}

Then, (\texttt{JDPEW}) can be reformulated as  (\texttt{JDPEW'}):
\begin{footnotesize}
\begin{align*}
  \min \quad &\sum_{j \in \mathcal{M}} \lambda_j u_{0 j} \hat{R}_j t_j+\sum_{i \in \mathcal{N}} \theta y_i 
  +\sum_{j \in \mathcal{M}} \lambda_j \sum_{i \in \mathcal{N}} \hat{V}_{i j} \left[\hat{R}_j x_{i j} - \hat{P}_{i j}\left(\sum_{h \in \mathcal{L}} z_{i h}^{*} x_{i j} d_h\right) + \hat{C}_{i j}x_{ij}\right] t_j \notag\\
  & \quad + \sum_{j \in \mathcal{M}} \lambda_j \sum_{i \in \mathcal{N}} \beta_j \hat{P}_{i j} \left[-\hat{R}_j \left(\sum_{h \in \mathcal{L}} z_{i h}^{*} x_{i j} d_h\right) + \hat{P}_{i j}\left(\sum_{h \in \mathcal{L}} z_{i h}^{*} x_{i j} d_h^2\right) - \left(\sum_{h \in \mathcal{L}} z_{i h}^{*} x_{i j} d_h\right) \hat{C}_{i j}\right] t_j 
\end{align*}
\end{footnotesize}
\vspace{-2.5em}
\begin{align*}
\text { s.t }  \quad &x_{i j} \leq y_i \hspace{18em}   \forall i \in \mathcal{N},~j \in \mathcal{M} \hspace{6em}\\
\quad &\sum_{i \in \mathcal{N}} \mathbb{I}_{i k} x_{i j} \geq 1 \hspace{15.5em} \forall k \in \mathcal{W},~j \in \mathcal{M} \hspace{6em} \\
\quad &y_i \in\{0,1\}, x_{i j} \in\{0,1\} \hspace{12em} \forall i \in \mathcal{N},~j \in \mathcal{M}.\hspace{6em}
\end{align*}

Let $g_{i j}=t_j x_{i j}$. Similarly, we can replace this equality with constraints (\ref{appdex:start})-(\ref{appdex:end}):
\begin{align}
& g_{i j} \geq t_j+\left(x_{i j}-1\right) t_j^U\quad &&  \forall i \in \mathcal{N},~j \in \mathcal{M} \label{appdex:start}\\
& g_{i j} \geq x_{i j} t_j^L\quad &&  \forall i \in \mathcal{N},~j \in \mathcal{M}  \\
& g_{i j} \leq t_j+\left(x_{i j}-1\right) t_j^L \quad &&  \forall i \in \mathcal{N},~j \in \mathcal{M}\\
& g_{i j} \leq x_{i j} t_j^U\quad &&  \forall i \in \mathcal{N},~j \in \mathcal{M},  \label{appdex:end}
\end{align}
where $t_j^L=1/(u_{ 0j}+\sum_{i \in \mathcal{N}} \hat{V}_{i j})$ and $t_j^U=1/u_{ 0j}$. 

Step 1 can finally be transformed into the following model that can be solved directly with commercial solvers:
\begin{footnotesize}
\begin{align*}
  \min \quad& \sum_{j \in \mathcal{M}} \lambda_j u_{0 j} \hat{R}_j t_j+\sum_{i \in \mathcal{N}} \theta y_i 
  +\sum_{j \in \mathcal{M}} \lambda_j \sum_{i \in \mathcal{N}} \hat{V}_{i j} \left[\hat{R}_j g_{i j} - \hat{P}_{i j}\left(\sum_{h \in \mathcal{L}} z_{i h}^{*} g_{i j} d_h\right) + \hat{C}_{i j}g_{ij}\right] \\
  & \quad + \sum_{j \in \mathcal{M}} \lambda_j \sum_{i \in \mathcal{N}} \beta_j \hat{P}_{i j} \left[-\hat{R}_j \left(\sum_{h \in \mathcal{L}} z_{i h}^{*} g_{i j} d_h\right) + \hat{P}_{i j}\left(\sum_{h \in \mathcal{L}} z_{i h}^{*} g_{i j} d_h^2\right) - \left(\sum_{h \in \mathcal{L}} z_{i h}^{*} g_{i j} d_h\right) \hat{C}_{i j}\right] 
\end{align*}
\end{footnotesize}
\vspace{-2.0em}
\begin{align*}
\text { s.t } \quad &x_{i j} \leq y_i &&  \forall i \in \mathcal{N},~j \in \mathcal{M} \qquad\qquad\qquad\qquad \\
&\sum_{i \in \mathcal{N}} \mathbb{I}_{i k} x_{i j} \geq 1 && \forall k \in \mathcal{W},~j \in \mathcal{M} \\
& t_j w_j\geq1\quad &&  \forall j \in \mathcal{M}\\
& w_j=u_{0 j}+\sum_{i \in \mathcal{N}}\left[\hat{V}_{i j} x_{i j}-\beta_j \hat{P}_{i j}\left(\sum_{h \in \mathcal{L}} z_{i h}^{*} x_{i j} d_h\right)\right] &&  \forall j \in \mathcal{M} \\
& g_{i j} \geq t_j+\left(x_{i j}-1\right) t_j^U\quad &&  \forall i \in \mathcal{N},~j \in \mathcal{M} \\
& g_{i j} \geq x_{i j} t_j^L\quad &&  \forall i \in \mathcal{N},~j \in \mathcal{M}  \\
& g_{i j} \leq t_j+\left(x_{i j}-1\right) t_j^L \quad &&  \forall i \in \mathcal{N},~j \in \mathcal{M}\\
& g_{i j} \leq x_{i j} t_j^U\quad &&  \forall i \in \mathcal{N},~j \in \mathcal{M}\\
&y_i \in\{0,1\},\, x_{i j} \in\{0,1\},\, t_j \geq 0,\, w_j \geq 0 && \forall i \in \mathcal{N},~j \in \mathcal{M}. \\
\end{align*}
\vspace{-4em}

\subsection{Reformulation of Step 2}
Let 
\begin{equation}
   t_j=\frac{1}{u_{ 0j}+\sum_{i \in \mathcal{N}}\left[\hat{V}_{i j}-\beta_j \hat{P}_{i j}\left(\sum_{h \in \mathcal{L}} z_{i h} d_h\right)\right] x_{i j}^{*}}\notag
\end{equation}
and 
\begin{equation}
   w_j=u_{ 0j}+\sum_{i \in \mathcal{N}}\left[\hat{V}_{i j}-\beta_j \hat{P}_{i j}\left(\sum_{h \in \mathcal{L}} z_{i h} d_h\right)\right] x_{i j}^{*}.\notag
\end{equation}
Then, model (\texttt{JDPEW}) can be reformulated as  (\texttt{JDPEW''}):
\begin{footnotesize}
\begin{align}
  \min \quad&\sum_{j \in \mathcal{M}} \lambda_j u_{ 0j} \hat{R}_j t_j+\sum_{i \in \mathcal{N}} \theta y_i^{*} 
  +\sum_{j \in \mathcal{M}} \lambda_j \sum_{i \in \mathcal{N}} \hat{V}_{i j} \left[\hat{R}_j x_{i j}^{*} - \hat{P}_{i j}\left(\sum_{h \in \mathcal{L}} z_{i h} x_{i j}^{*} d_h\right) + \hat{C}_{i j}x_{ij}^{*}\right] t_j \notag\\
  & \quad + \sum_{j \in \mathcal{M}} \lambda_j \sum_{i \in \mathcal{N}} \beta_j \hat{P}_{i j} \left[-\hat{R}_j \left(\sum_{h \in \mathcal{L}} z_{i h} x_{i j}^{*} d_h\right) + \hat{P}_{i j}\left(\sum_{h \in \mathcal{L}} z_{i h} x_{i j}^{*} d_h^2\right) - \left(\sum_{h \in \mathcal{L}} z_{i h} x_{i j}^{*} d_h\right) \hat{C}_{i j}\right] t_j \notag
\end{align}
\end{footnotesize}
\vspace{-2.0em}
\begin{align}
  \quad\quad\text { s.t } \quad & \sum_{h \in \mathcal{L}} z_{i h}=1 \hspace{19em} \forall i \in \mathcal{N} \hspace{4em}\notag \\
  \quad &\left(\sum_{h \in \mathcal{L}} z_{i h} d_h-\sum_{h^{\prime} \in L} z_{i^{\prime} h^{\prime}} d_{h^{\prime}}\right)\left(\sum_{k \in \mathcal{W}} \mathbb{I}_{i k}-\sum_{k \in \mathcal{W}} \mathbb{I}_{i^{\prime} k}\right) \leq 0 \hspace{1em}  \forall i,~i^{\prime} \in \mathcal{N} \hspace{4em} \notag \\
  \quad & p_{ij}=\left(\sum_{k\in \mathcal{W}}\mathbb{I}_{ik}P_{kj}^0\right)\left(\sum_{h\in \mathcal{L}}z_{ih}d_{h}\right)\hspace{9.3em}  \forall i \in \mathcal{N},~j \in \mathcal{M} \notag\\
  \quad &z_{ih} \in\{0,1\}  \hspace{19 em}  \forall i \in \mathcal{N},~h \in \mathcal{L}. \hspace{7em}\notag
\end{align}
Let $g_{ij h}=t_j z_{i h}$. Similarly, we can replace this equality with constraints (\ref{appendix:start1})-(\ref{appendix:end1}):
\begin{align}
& g_{ij h} \geq t_j+\left(z_{i h}-1\right) t_j^U\quad &&  \forall i \in \mathcal{N},~j \in \mathcal{M}\label{appendix:start1} \\
& g_{i jh} \geq z_{i h} t_j^L\quad &&  \forall i \in \mathcal{N},~j \in \mathcal{M}  \\
& g_{i jh} \leq t_j+\left(z_{i h}-1\right) t_j^L \quad &&  \forall i \in \mathcal{N},~j \in \mathcal{M}\\
& g_{i jh} \leq z_{i h} t_j^U\quad &&  \forall i \in \mathcal{N},~j \in \mathcal{M},  \label{appendix:end1}
\end{align}
where $t_j^L=1/(u_{ 0j}+\sum_{i \in \mathcal{N}} \hat{V}_{i j})$ and $t_j^U=1/u_{ 0j}$. 

Finally, Step 2 can be transformed into the following model that can be solved directly with commercial solvers:
\begin{footnotesize}
\begin{align*}
  \min \quad& \sum_{j \in \mathcal{M}} \lambda_j u_{ 0j} \hat{R}_j t_j+\sum_{i \in \mathcal{N}} \theta y_i^{*} 
  +\sum_{j \in \mathcal{M}} \lambda_j \sum_{i \in \mathcal{N}} \hat{V}_{i j} \left[\hat{R}_j x_{i j}^{*}t_j - \hat{P}_{i j}\left(\sum_{h \in \mathcal{L}} g_{i jh} x_{i j}^{*} d_h\right) + \hat{C}_{i j}x_{ij}^{*}t_{j}\right] +\qquad\\
  &\sum_{j \in \mathcal{M}} \lambda_j \sum_{i \in \mathcal{N}} \beta_j \hat{P}_{i j} \left[-\hat{R}_j \left(\sum_{h \in \mathcal{L}} g_{ij h} x_{i j}^{*} d_h\right) + \hat{P}_{i j}\left(\sum_{h \in \mathcal{L}} g_{ij h} x_{i j}^{*} d_h^2\right) - \left(\sum_{h \in \mathcal{L}} g_{ij h} x_{i j}^{*} d_h\right) \hat{C}_{i j}\right]
\end{align*}
\end{footnotesize}
\vspace{-2.0em}
\begin{align*}
\text { s.t } \quad & \sum_{h \in \mathcal{L}} z_{i h}=1 &&  \forall 
 i \in \mathcal{N} \\
&\left(\sum_{h \in \mathcal{L}} z_{i h} d_h-\sum_{h^{\prime} \in L} z_{i^{\prime} h^{\prime}} d_{h^{\prime}}\right)\left(\sum_{k \in \mathcal{W}} \mathbb{I}_{i k}-\sum_{k \in \mathcal{W}} \mathbb{I}_{i^{\prime} k}\right) \leq 0 &&  \forall i \in \mathcal{N} , i^{\prime} \in N \\
& p_{ij}=\left(\sum_{k\in \mathcal{W}}\mathbb{I}_{ik}P_{kj}^0\right)\left(\sum_{h\in \mathcal{L}}z_{ih}d_{h}\right)&&  \forall i \in \mathcal{N},~j \in \mathcal{M} \\
& t_j w_j\geq1\quad &&  \forall j \in \mathcal{M}\\
& w_j=u_{ 0j}+\sum_{i \in \mathcal{N}}\left[\hat{V}_{i j} x_{i j}^{*}-\beta_j \hat{P}_{i j}\left(\sum_{h \in \mathcal{L}} z_{i h} x_{i j}^{*} d_h\right)\right] &&  \forall j \in \mathcal{M} \\
& g_{ij h} \geq t_j+\left(z_{i h}-1\right) t_j^U\quad &&  \forall i \in \mathcal{N} ,j \in \mathcal{M} , h \in \mathcal{L}\\
& g_{i jh} \geq z_{i h} t_j^L\quad &&  \forall i \in \mathcal{N} ,j \in \mathcal{M} , h \in \mathcal{L} \\
& g_{i jh} \leq t_j+\left(z_{i h}-1\right) t_j^L \quad &&  \forall i \in \mathcal{N} ,j \in \mathcal{M}, h \in \mathcal{L}\\
& g_{i jh} \leq z_{i h} t_j^U\quad &&  \forall i \in \mathcal{N} ,j \in \mathcal{M} , h \in \mathcal{L}\\
&z_{ih} \in\{0,1\} &&  \forall i \in \mathcal{N} ,h \in \mathcal{L}. \\
\end{align*}

\vspace{-4em}

\section{Benchmark Models} \label{app:benchmark}

\subsection{Benchmark 1: Consistent EW Design}
Let $z_{ih}^{0}$ denote the initial value of $z_{ih}$, where $z_{ih}^{0}$ is calculated as follows:
\[\begin{aligned}
& z_{i h}^{0}= \begin{cases}1 & \text { if } h=\operatorname{argmax}_{k\in W}\{k:\mathbb{I}_{ik}=1\}, \\
0 & \text { otherwise, }
\end{cases} \\
\end{aligned}\]
 for all $i\in \mathcal{N}$ and $h\in \mathcal{L}$. Then, the corresponding price is set to $p_{ij}^0=(\sum_{k\in \mathcal{W}}\mathbb{I}_{ik}P_{kj}^0)(\sum_{h\in \mathcal{L}}z_{ih}^0 d_{h})$.

The model of consistent warranty design for all segments is as follows:
\begin{align*}
\max \quad& \sum_{j \in \mathcal{M}} \lambda_j q_{i j}\left(\mathbf{p}_{j}^{0}, \mathbf{x}\right) \cdot \sum_{i \in \mathcal{N}}\left[\left(\sum_{k \in \mathcal{W}} \mathbb{I}_{i k} P_{k j}\right)\left(\sum_{h \in \mathcal{L}} z_{i h}^{0} d_h\right)-\sum_{k \in \mathcal{W}} \mathbb{I}_{i k} f_{k j} c_{kj}\right]-\sum_{i \in \mathcal{N}} \theta x_i 
\end{align*}
\vspace{-2.0em}
\begin{align*}
\text { s.t. } \quad& \sum_{i \in \mathcal{N}} \mathbb{I}_{i k} x_i \geq 1   &&\forall k \in \mathcal{W}\quad\quad   \\
& x_i \in\{0,1\}  &&\forall i \in \mathcal{N}, 
\end{align*}
where $q_{i j}(\mathbf{p}_{j}^{0}, \mathbf{x})$ is group-$j$ customers' choice probability of EW contract $i$. The structure of the function remains similar to what is described in Equation~\eqref{choice prob}, with $p_{ij}$ replaced by $p_{ij}^{0}$. 

Through a series of transformations of the above model, the final model to be solved after reformulation is as follows:
\begin{align*}
  \min \quad &\sum_{j \in \mathcal{M}} \lambda_j u_{ 0j} \hat{R}_j t_j+\sum_{i \in \mathcal{N}} \theta x_i +\sum_{j \in \mathcal{M}} \lambda_j \sum_{i \in \mathcal{N}} \hat{V}_{i j} \Bigg[\hat{R}_j g_{ij} - \hat{P}_{i j}\left(\sum_{h \in \mathcal{L}}z_{ih}^{0} g_{i j} d_h\right) + \hat{C}_{i j}g_{ij} \Bigg]\\
  & + \sum_{j \in \mathcal{M}} \lambda_j \sum_{i \in \mathcal{N}} \beta_j \hat{P}_{i j} \Bigg[-\hat{R}_j \left(\sum_{h \in \mathcal{L}}z_{ih}^{0} g_{i j} d_h\right) + \hat{P}_{i j}\left(\sum_{h \in \mathcal{L}} z_{ih}^{0}g_{i j} d_h^2\right) - \left(\sum_{h \in \mathcal{L}} z_{ih}^{0}g_{i j} d_h\right) \hat{C}_{i j} \Bigg]
\end{align*}
\vspace{-2.0em}
\begin{align*}
\text{s.t.}\quad &\sum_{i \in \mathcal{N}} \mathbb{I}_{i k} x_i \geq 1 &&  \forall  k \in \mathcal{W} \\
&t_j w_j \geq 1 &&  \forall j \in \mathcal{M} \\
&w_j=u_{0j}+\sum_{i \in \mathcal{N}}\left[\hat{V}_{i j} x_i-\beta_j \hat{P}_{i j}\left(\sum_{h \in \mathcal{L}} z_{i h}^{0}x_{i} d_h\right)\right] &&  \forall j \in \mathcal{M} \\
&g_{i j} \geq t_j+\left(x_i-1\right) t_j^U &&  \forall i \in \mathcal{N},j \in \mathcal{M} \\
&g_{i j} \geq x_i t_j^L &&  \forall i \in \mathcal{N}, j \in \mathcal{M} \\
&g_{i j} \leq t_j+\left(x_i-1\right) t_j^L &&  \forall i \in \mathcal{N}, j \in \mathcal{M} \\
&g_{i j} \leq x_i t_j^U &&  \forall i \in \mathcal{N}, j \in \mathcal{M}\\
&x_i \in\{0,1\},\,t_j \geq 0,\, w_j \geq 0 &&  \forall i \in \mathcal{N},j \in \mathcal{M}. \\
\end{align*}

\subsection{Benchmark 2: Personalized EW Design}
This model is the same as Step 1 in Section~\ref{Subsec:twostage}:
\begin{align*}
\max \quad&\sum_{j \in \mathcal{M}} \lambda_j  \sum_{i \in \mathcal{N}}q_{i j}\left(\mathbf{p}_{j}^{0}, \mathbf{x}_j\right)\left[\left(\sum_{k \in \mathcal{W}} \mathbb{I}_{i k} P_{k j}\right)\left(\sum_{h \in \mathcal{L}} z_{i h}^{0} d_h\right)-\sum_{k \in \mathcal{W}} \mathbb{I}_{i k} f_{k j} c_{k j}\right]-\sum_{i \in \mathcal{N}} \theta y_i \\
\text { s.t. } \quad &x_{i j} \leq y_i \hspace{18em}  \forall i \in \mathcal{N} ,j \in \mathcal{M} \\
 &\sum_{i \in \mathcal{N}} \mathbb{I}_{i k} x_{i j} \geq 1 \hspace{15.5em} \forall k \in \mathcal{W} ,j \in \mathcal{M} \\
 &y_i \in\{0,1\},\, x_{i j} \in\{0,1\} \hspace{12em} \forall i \in \mathcal{N} ,j \in \mathcal{M}. 
\end{align*}

Through a series of transformation of the above model, the final model to be solved after reformulation is as follows:
\begin{small}
\begin{align*}
    \min \quad& \sum_{j \in \mathcal{M}} \lambda_j u_{ 0j} \hat{R}_j t_j+\sum_{i \in \mathcal{N}} \theta y_i 
    +\sum_{j \in \mathcal{M}} \lambda_j \sum_{i \in \mathcal{N}} \hat{V}_{i j} \left[\hat{R}_j g_{i j} - hat{P}_{i j}\left(\sum_{h \in \mathcal{L}} z_{i h}^{*} g_{i j} d_h\right) + \hat{C}_{i j}g_{ij}\right] \\
    &\quad + \sum_{j \in \mathcal{M}} \lambda_j \sum_{i \in \mathcal{N}} \beta_j \hat{P}_{i j} \left[-\hat{R}_j \left(\sum_{h \in \mathcal{L}} z_{i h}^{*} g_{i j} d_h\right) + \hat{P}_{i j}\left(\sum_{h \in \mathcal{L}} z_{i h}^{*} g_{i j} d_h^2\right) - \left(\sum_{h \in \mathcal{L}} z_{i h}^{*} g_{i j} d_h\right) \hat{C}_{i j}\right] 
\end{align*}
\end{small}
\vspace{-2.5em}
\begin{align*}
\qquad\text { s.t. } \quad &x_{i j} \leq y_i &&  \forall i \in \mathcal{N} ,j \in \mathcal{M} \qquad\qquad\qquad\qquad \\
\quad &\sum_{i \in \mathcal{N}} \mathbb{I}_{i k} x_{i j} \geq 1 && \forall k \in \mathcal{W} ,j \in \mathcal{M} \\
& t_j w_j\geq1\quad &&  \forall j \in \mathcal{M}\\
& w_j=u_{ 0j}+\sum_{i \in \mathcal{N}}\left[\hat{V}_{i j} x_{i j}-\beta_j \hat{P}_{i j}\left(\sum_{h \in \mathcal{L}} z_{i h}^{*} x_{i j} d_h\right)\right] &&  \forall j \in \mathcal{M} \\
& g_{i j} \geq t_j+\left(x_{i j}-1\right) t_j^U\quad &&  \forall i \in \mathcal{N} ,j \in \mathcal{M} \\
& g_{i j} \geq x_{i j} t_j^L\quad &&  \forall i \in \mathcal{N} ,j \in \mathcal{M}  \\
& g_{i j} \leq t_j+\left(x_{i j}-1\right) t_j^L \quad &&  \forall i \in \mathcal{N} ,j \in \mathcal{M}\\
& g_{i j} \leq x_{i j} t_j^U\quad &&  \forall i \in \mathcal{N} ,j \in \mathcal{M}\\
&y_i \in\{0,1\},\, x_{i j} \in\{0,1\},\, t_j \geq 0,\, w_j \geq 0  && \forall i \in \mathcal{N} ,j \in \mathcal{M}.  
\end{align*}

\subsection{Benchmark 3: Bundle Pricing with Consistent Design}

This is a commonly used method in reality to recommend the same contract for all different groups of products. This means that $x_{ij}=x_{ij'}$ for all $j\neq j'$. For convenience, we will denote it as $x_{i}$. Here, the decision variables are the contracts recommended for products $x_{i}$, as well as the discount prices for each contract $z_{ih}$. The optimization problem in this scenario is
\begin{align*}
    \max \quad \sum_{j \in \mathcal{M}} \lambda_j q_{i j}\left(\mathbf{p}_{j}^{0}, \mathbf{x}\right) \cdot \sum_{i \in \mathcal{N}}\left[\left(\sum_{k \in \mathcal{W}} \mathbb{I}_{i k} P_{k j}\right)\left(\sum_{h \in \mathcal{L}} z_{i h} d_h\right)-\sum_{k \in \mathcal{W}} \mathbb{I}_{i k} f_{k j} c_{k j}\right]-\sum_{i \in \mathcal{N}} \theta x_i 
\end{align*}
\vspace{-2.5em}
\begin{align*}
    \text { s.t. } \quad& \sum_{i \in \mathcal{N}} \mathbb{I}_{i k} x_i \geq 1 && \forall k \in \mathcal{W} \\
    & \sum_{h \in \mathcal{L}} z_{i h}=1 && \forall i \in \mathcal{N} \\
    & \left(\sum_{h \in \mathcal{L}} z_{i h} d_h-\sum_{h^{\prime} \in L} z_{i^{\prime} h^{\prime}} d_{h^{\prime}}\right)\left(\sum_{k \in \mathcal{W}} \mathbb{I}_{i k}-\sum_{k \in \mathcal{W}} \mathbb{I}_{i^{\prime}k}\right) \leq 0 &&\forall i \in \mathcal{N}, ~ i \in \mathcal{N} \\
    & p_{ij}=\left(\sum_{k\in \mathcal{W}}\mathbb{I}_{ik}P_{kj}^0\right)\left(\sum_{h\in \mathcal{L}}z_{ih}d_{h}\right)&&  \forall i \in \mathcal{N},~j \in \mathcal{M} \\
    & x_i \in\{0,1\},\, z_{i h} \in\{0,1\} && \forall i \in \mathcal{N},  h \in \mathcal{L}. 
\end{align*}

Through a series of transformations of the above model, the final model to be solved after reformulation is as follows:
\begin{align*}
  \min \quad& \sum_{j \in \mathcal{M}} \lambda_j u_{ 0j} \hat{R}_j t_j+\sum_{i \in \mathcal{N}} \theta x_i +\sum_{j \in \mathcal{M}} \lambda_j \sum_{i \in \mathcal{N}} \hat{V}_{i j} \Bigg[\hat{R}_j g_{ij} - \hat{P}_{i j}\left(\sum_{h \in \mathcal{L}} o_{i j h} d_h\right) + \hat{C}_{i j}g_{ij} \Bigg]\\
  & + \sum_{j \in \mathcal{M}} \lambda_j \sum_{i \in \mathcal{N}} \beta_j \hat{P}_{i j} \Bigg[-\hat{R}_j \left(\sum_{h \in \mathcal{L}} o_{i j h} d_h\right) + \hat{P}_{i j}\left(\sum_{h \in \mathcal{L}} o_{i j h} d_h^2\right)- \left(\sum_{h \in \mathcal{L}} o_{i j h} d_h\right) \hat{C}_{i j} \Bigg]
\end{align*}
\begin{align*}
\text{s.t.}\quad &\sum_{i \in \mathcal{N}} \mathbb{I}_{i k} x_i \geq 1 &&  \forall  k \in \mathcal{W} \\
&\sum_{h \in \mathcal{L}} z_{i h}=1 &&  \forall i \in \mathcal{N} \\
&\left(\sum_{h \in \mathcal{L}} z_{i h} d_h-\sum_{h^{\prime} \in L} z_{i^{\prime} h^{\prime}} d_{h^{\prime}}\right)\left(\sum_{k \in \mathcal{W}} \mathbb{I}_{i k}-\sum_{k^{\prime} \in W} \mathbb{I}_{i k^{\prime}}\right) \leq 0 &&  \forall  i, i^{\prime} \in N \\
&p_{ij}=\left(\sum_{k\in \mathcal{W}}\mathbb{I}_{ik}P_{kj}^0\right)\left(\sum_{h\in \mathcal{L}}z_{ih}d_{h}\right)&&  \forall i \in \mathcal{N},~j \in \mathcal{M} \\
&t_j w_j \geq 1 &&  \forall j \in \mathcal{M} \\
&w_j=u_{0j}+\sum_{i \in \mathcal{N}}\left[\hat{V}_{i j} x_i-\beta_j \hat{P}_{i j}\left(\sum_{h \in \mathcal{L}} s_{i h} d_h\right)\right] && \forall j \in \mathcal{M} \\
&s_{i h} \geq z_{i h}+\left(x_i-1\right) &&  \forall i \in \mathcal{N},  h \in \mathcal{L} \\
&s_{i h} \geq 0 &&  \forall i \in \mathcal{N}, h \in \mathcal{L} \\
&s_{i h} \leq z_{i h} &&  \forall i \in \mathcal{N}, h \in \mathcal{L} \\
&s_{i h} \leq x_i && \forall i \in \mathcal{N}, h \in \mathcal{L} \\
&g_{i j} \geq t_j+\left(x_i-1\right) t_j^U &&  \forall i \in \mathcal{N}, j \in \mathcal{M} \\
&g_{i j} \geq x_i t_j^L &&  \forall i \in \mathcal{N}, j \in \mathcal{M} \\
&g_{i j} \leq t_j+\left(x_i-1\right) t_j^L &&  \forall i \in \mathcal{N}, j \in \mathcal{M} \\
&g_{i j} \leq x_i t_j^U &&  \forall i \in \mathcal{N}, j \in \mathcal{M} \\
&o_{i j h} \geq s_{i h} t_j^L &&  \forall i \in \mathcal{N}, j \in \mathcal{M}, h \in \mathcal{L} \\
&o_{i j h} \geq t_j+\left(s_{i h}-1\right) t_j^U &&  \forall i \in \mathcal{N}, j \in \mathcal{M}, h \in \mathcal{L} \\
&o_{i j h} \leq t_j+\left(s_{i h}-1\right) t_j^L && \forall i \in \mathcal{N}, j \in \mathcal{M}, h \in \mathcal{L} \\
&o_{i j h} \leq s_{i h} t_j^U && \forall i \in \mathcal{N}, j \in \mathcal{M}, h \in \mathcal{L}\\
&x_i \in\{0,1\},\, z_{i h} \in\{0,1\},\, t_j \geq 0,\, w_j \geq 0 &&  \forall i \in \mathcal{N}, j \in \mathcal{M}, h \in \mathcal{L}.
\end{align*}

\section{Proofs} \label{app:proof}

\begin{proof}[Proof of Proposition~\ref{Prop1}]
To demonstrate the complexity of the original Problem (\texttt{JDPEW}), we first simplify it to a simpler problem. If the simpler problem is NP-hard, then the original problem is also NP-hard. Specifically, when the pricing decision is ignored, the original problem can be simplified to a simpler problem as follows:
\begin{align*}
\max \quad&\sum_{j \in M} \lambda_j \frac{\sum_{i \in N} u_{i j}r_{i j}  x_{i j}}{u_{ 0j}+\sum_{i \in N} u_{i j} x_{i j}}-\sum_{i \in N} \theta_i y_i \\
 \text { s.t. } \quad& x_{i j} \leq y_i \hspace{9.2em}  \forall  i \in N , j \in M \quad \text { (CAP) } \\
& x_{i j} \in\{0,1\}, y_i \in\{0,1\} \hspace{3em}  \forall  i \in N , j \in M.
\end{align*}

Then we show that Problem (\texttt{CAP}) is NP-hard even though there are only two automobile groups. Let $m=2$, $\lambda_1=\lambda_2=\frac{1}{2}$ , Problem (\texttt{CAP}) equals to:
\begin{align*}
\max \quad& \frac{1}{2} \frac{\sum_{i \in N} u_{i 1} r_{i 1} x_{i 1}}{u_{01}+\sum_{i \in N} u_{i 1} x_{i 1}}+\frac{1}{2} \frac{\sum_{i \in N} u_{i 2} r_{i 2} x_{i 2}}{u_{02}+\sum_{i \in N} u_{i 2} x_{i 2}}-\sum_{i \in N} \theta_i y_i \\
\text { s.t. } \quad & x_{i 1} \leq y_i \hspace{13em}  \forall  i \in N \\
& x_{i 2} \leq y_i \hspace{13em}  \forall  i \in N \\
& x_{i 1} \in\{0,1\}, x_{i 2} \in\{0,1\}, y_i \in\{0,1\} \hspace{1.6em}  \forall  i \in N.
\end{align*}

Then, we transform an arbitrary instance of Partition problem, which is a well-known NP-complete problem \citep[see, e.g.,][]{garey1979computers}, and is equivalent (\texttt{CAP}) problem. The Partition problem is defined as follows.

\textbf{INPUTS:} Set of items indexed by $1,2, \cdots, n$ and the size $c_i \in Z_{+}$ associated with each item $i$

\textbf{Questions:} Is there a subset $S \subseteq\{1,2, \cdots, n\}$ such $\sum_{i \in S} c_i=\sum_{i \in\{1,2, \cdots, n\} \backslash  S} c_i ?$

Let $C=\frac{1}{2} \sum_{i \in N} c_i$. Note that $\sum_{i \in S} c_i=\sum_{j \in\{1,2, \cdots, n\} \backslash S} c_i$ if and only if $\sum_{i \in S} c_i=C$. Therefore, we assume that $C \in Z_{+}$ without loss of generality, and construct an instance of the (\texttt{CAP}) problem with parameter realizations as follows:
\begin{align}
& r_{i 1}=1 \text { and } r_{i 2}=3 \text { for all } i \in N \notag\\
& u_{01}=1, u_{02}=2, u_{i 1}=c_i \text { and } u_{i 2}=2 c_i \text { for all } i \in N \label{specified} \\
& \theta_i=\frac{(C+1)^2}{2} c_i \text { for all } i \in N. \notag
\end{align}

First, we show that if $x_{i 1}^*, x_{i 2}^*$ and $y_i^*$ is the optimal solution of \texttt{CAP} specified as in Equation (\ref{specified}), then $x_{i 1}^*=x_{i 2}^*=y_i^*$. We will use proof by contradiction, and three cases should be considered as follows.
\begin{itemize}
    \item Case 1: We assume that there exists an $i^{\prime} \in N$, such that $y_{i^{\prime} 1}^*=1, x_{i^{\prime} 1}^*=0$ and $x_{i^{\prime} 2}^*=1$. Let $\pi^*$ represent the objective function value corresponding to the optimal solution, and let $\pi^{\prime}$ represent the objective function value when $x_{i^{\prime} 1}^*=0$ is changed to $x_{i^{\prime} 1}=1$, while the values of the other solutions remain unchanged. Let $\Delta_1$ represent the difference between the two objective function values $\pi^{\prime}$ and $\pi^*$. Then, we have:
\begin{align*}
\Delta_1 & =\pi^{\prime}-\pi^* \\
& =\frac{1}{2} \frac{\sum_{i \in N \backslash\{i^{\prime}\}} u_{i 1} x_{i 1}^*+u_{i^{\prime} 1}}{u_{01}+\sum_{i \in N \backslash\left\{i'\right\} }u_{i 1} x_{i 1}^*+u_{i^{\prime} 1}}-\frac{1}{2} \frac{\sum_{i \in N \backslash\{i'\}} u_{i 1} x_{i 1}^*}{u_{01}+\sum_{i \in N \backslash\{i^{\prime}\}} u_{i 1} x_{i 1}^*} \\
& =\frac{1}{2} \frac{u_{i^{\prime} 1} u_{01}}{\left(u_{01}+\sum_{i \in N \backslash\{i^{\prime}\}} u_{i 1} x_{i 1}^*+u_{i^{\prime} 1}\right)\left(u_{01}+\sum_{i \in N \backslash\{i^{\prime}\}} u_{i 1} x_{i 1}^*\right)} \geq 0.
\end{align*}
This contradicts the fact that $x_{i 1}^*, x_{i 2}^*$ and $y_i^*$ is the optimal solution of problem (\texttt{CAP}) specified as in Equation (\ref{specified}).
\item Case 2: We assume that there exists an $i^{\prime} \in N$, such that $y_{i^{\prime} 1}^*=1, x_{i\prime 1}^*=1$ and $x_{i^{\prime} 2}^*=0$. Similar to the proof in Case 1, if we change $x_{i^{\prime} 2}^*=0$ to $x_{i^{\prime} 2}^*=1$, we can obtain a larger objective function value.

\item Case 3: We assume that there exists $i^{\prime}, i^{\prime \prime} \in N$, such that $y_{i^{\prime} 1}^*=y_{i^{\prime\prime} 1}^*=1, x_{i^{\prime} 1}^*=0$, and $x_{i^{\prime\prime} 2}^*=0$. Let $\pi^{* *}$ represent the objective function value corresponding to the optimal solution. Let $\pi^{\prime\prime}$ represent the objective function value when $x_{i^{\prime} 1}^*=0$ is changed to $x_{i^{\prime} 1}=1$ and $x_{i^{\prime\prime} 2}^*=0$ is changed to $x_{i^{\prime\prime} 2}^*=1$, while the values of other solutions remain unchanged. Let $\Delta_2$ represent the difference between the two objective function values $\pi^{\prime\prime}$ and $\pi^{* *}$. Then, we have:
\begin{align*}
\Delta_2 & =\pi^*-\pi^{* *} \\
& =\frac{1}{2} \frac{\sum_{i \in N \backslash\{i^{\prime}\}} u_{i 1} x_{i 1}^*+u_{i^{\prime} 1}}{u_{01}+\sum_{i \in N \backslash\{i^{\prime}\}} u_{i 1} x_{i 1}^*+u_{i^{\prime} 1}}-\frac{1}{2} \frac{\sum_{i \in N \backslash \{i^{\prime}\}} u_{i 1} x_{i 1}^*}{u_{01}+\sum_{i \in N \backslash \{i^{\prime}\}} u_{i 1} x_{i 1}^*}\\
&\quad +\frac{1}{2} \frac{\sum_{i \in N \backslash\{i^{\prime\prime}\}} 3 u_{i 2} x_{i 2}^*+3 u_{i^{\prime\prime} 2}}{u_{02}+\sum_{i \in N \backslash\{i^{\prime\prime}\}} u_{i 2} x_{i 2}^*+u_{i^{\prime\prime} 2}}-\frac{1}{2} \frac{\sum_{i \in N \backslash\{i^{\prime\prime}\}} 3 u_{i 2} x_{i 2}^*}{u_{02}+\sum_{i \in N \backslash\{i^{\prime\prime}\}} u_{i 2} x_{i 2}^*} \\
& =\frac{1}{2} \frac{u_{i^{\prime} 1} u_{01}}{\left(u_{01}+\sum_{i \in N \backslash\{i^{\prime}\}} u_{i 1} x_{i 1}^*+u_{i^{\prime} 1}\right)\left(u_{01}+\sum_{i \in N \backslash\{i^{\prime}\}} u_{i 1} x_{i 1}^*\right)} \\
&\quad +\frac{1}{2} \frac{3 u_{i^{\prime\prime} 2} u_{02}}{\left(u_{02}+\sum_{i \in N \backslash\{i^{\prime\prime}\}} u_{i 2} x_{i 2}^*+u_{i^{\prime\prime} 2}\right)\left(u_{02}+\sum_{i \in N \backslash\{i^{\prime\prime}\}} u_{i 2} x_{i 2}^*\right)}\geq 0.
\end{align*}
This contradicts the fact that $x_{i 1}^*, x_{i 2}^*$ and $y_i^*$ is the optimal solution of Problem (\texttt{CAP}) specified as in (\ref{specified}).
In summary, we obtain that if $x_{i 1}^*, x_{i 2}^*$ and $y_i^*$ is the optimal solution of (\texttt{CAP}) specified as in (\ref{specified}), then $x_{i 1}^*=x_{i 2}^*=y_i^*$. Therefore, we can express the instance of (\texttt{CAP}) specified as in (\ref{specified}) as follows:
\begin{align*}
\max &\quad \frac{1}{2} \frac{\sum_{i \in N} u_{i 1} y_i}{u_{01}+\sum_{i \in N} c_i y_i}+\frac{1}{2} \frac{\sum_{i \in N} 3 u_{i 2} y_i}{u_{02}+\sum_{i \in N} 2 u_{i 2} y_i}-\sum_{i \in N} \theta_i y_i \\
\text { s.t. } &\quad y_i \in\{0,1\} \hspace{10em} \text { for all } i \in N.
\end{align*}

\end{itemize}

Next, we set the target profit as $K=(2 C^2-2 C)/(C+1)^2$, and show the Partition problem has a solution if and only if the optimal value of the instance of (\texttt{CAP}) specified as in Equation (\ref{specified}) is $K$. Let
$$
F(z):=\frac{1}{2} \frac{z}{1+z}+\frac{1}{2} \frac{6 z}{2+2 z}-\frac{2 z}{(C+1)^2}.
$$
Then, we obtain that
$$
\begin{aligned}
& \max \left\{\left.\frac{1}{2} \frac{\sum_{i \in N} u_{i 1} y_i}{u_{01}+\sum_{i \in N} u_{i 1} y_i}+\frac{1}{2} \frac{\sum_{i \in N} 3u_{i 2} y_i}{u_{02}+\sum_{i \in N} u_{i 2} y_i}-\sum_{i \in N} \theta_i y_i \right\rvert\, \begin{array}{l}
y_i \in\{0,1\} \quad \text { for all } i \in N
\end{array}\right\} \\
& =\max \left\{F(z) \left\lvert\, \begin{array}{l}
z=\sum_{i \in N} c_i y_i \\
y_i \in\{0,1\} \quad \text { for all } i \in N \text { and } j \in\{1,2\}
\end{array}\right.\right\}.
\end{aligned}
$$
The derivative of $F(z)$ is given as follows,
$$
F^{\prime}(z)=\frac{2}{(1+z)^2}-\frac{2}{(C+1)^2},
$$
which is strictly positive over $[0, \mathrm{C})$ and $(C, \infty)$. Therefore, $F(\cdot)$ has a unique maximum at $C$, that is,
$$
F(z) \leq F(C)=\frac{2 C^2-2 C}{(C+1)^2} \text { for } z \in[0, C] \cup(C, \infty).
$$
Hence, we obtain
$$
\begin{aligned}
& \max \left\{\left.\frac{1}{2} \frac{\sum_{i \in N} u_{i 1} y_i}{u_{01}+\sum_{i \in N} u_{i 1} y_i}+\frac{1}{2} \frac{\sum_{i \in N} 3 u_{i 2} y_i}{u_{02}+\sum_{i \in N} u_{i 2} y_i}-\sum_{i \in N} \theta_i y_i \right\rvert\, y_i \in\{0,1\} \quad \text { for all } i \in N\right\} \\
& =\max \left\{F(z) \left\lvert\, \begin{array}{l}
z=\sum_{i \in N} c_i y_i \\
y_i \in\{0,1\} \quad \text { for all } i \in N \text { and } j \in\{1,2\}
\end{array}\right.\right\} \leq F(C)=K.
\end{aligned}
$$

In other words, there exists assortment $S=\left\{i \mid y_i=1, i \in N\right\}$ whose objective value is K if and only if the inequality holds as equality, where the latter is equivalent to that there exists a subset $S$ of $\{1,2, \cdots, n\}$ such that $\sum_{i \in S} c_i=C$. This completes the proof.

\end{proof}

\begin{proof}[Proof of Proposition~\ref{prop2}]
     This is obvious based on the transformation process in Section~\ref{Subsec:socp}.
\end{proof}

 \begin{proof}[Proof Proposition~\ref{prop3}]
Due to the finite number of dimensions and feasible values for each decision variable, the value of the objective function $g(\mathbf{X},\mathbf{y},\mathbf{P},\mathbf{Z})$ can only take finite many different values. 

According to Algorithm~\ref{alg:algorithm1}, as each iteration in Step 1 and Step 2 can monotonically increase the objective function, then it must become constant based on a bounded integer set of decision variables, as their potential feasible values are finite. 
 \end{proof}

 \begin{proof}[Proof of Proposition~\ref{prop4}]
  According to the equivalent reformulation for the sub-optimization problem in Step 1 and Step 2 in Appendix~\ref{app:reform}, the original optimization problem can be transformed into a model with existing solvers, and the optimal solution can be guaranteed accordingly. 
 Then it is always true that the optimization problem in Step 1 and Step 2 will achieve the optimal solution within each iteration.
 \end{proof}

\section{Additional Numerical Analyses} \label{MoreNum}

\subsection{Solution Illustration}

\begin{figure}[ht!]
    \centering
    \includegraphics[width=0.98\linewidth]{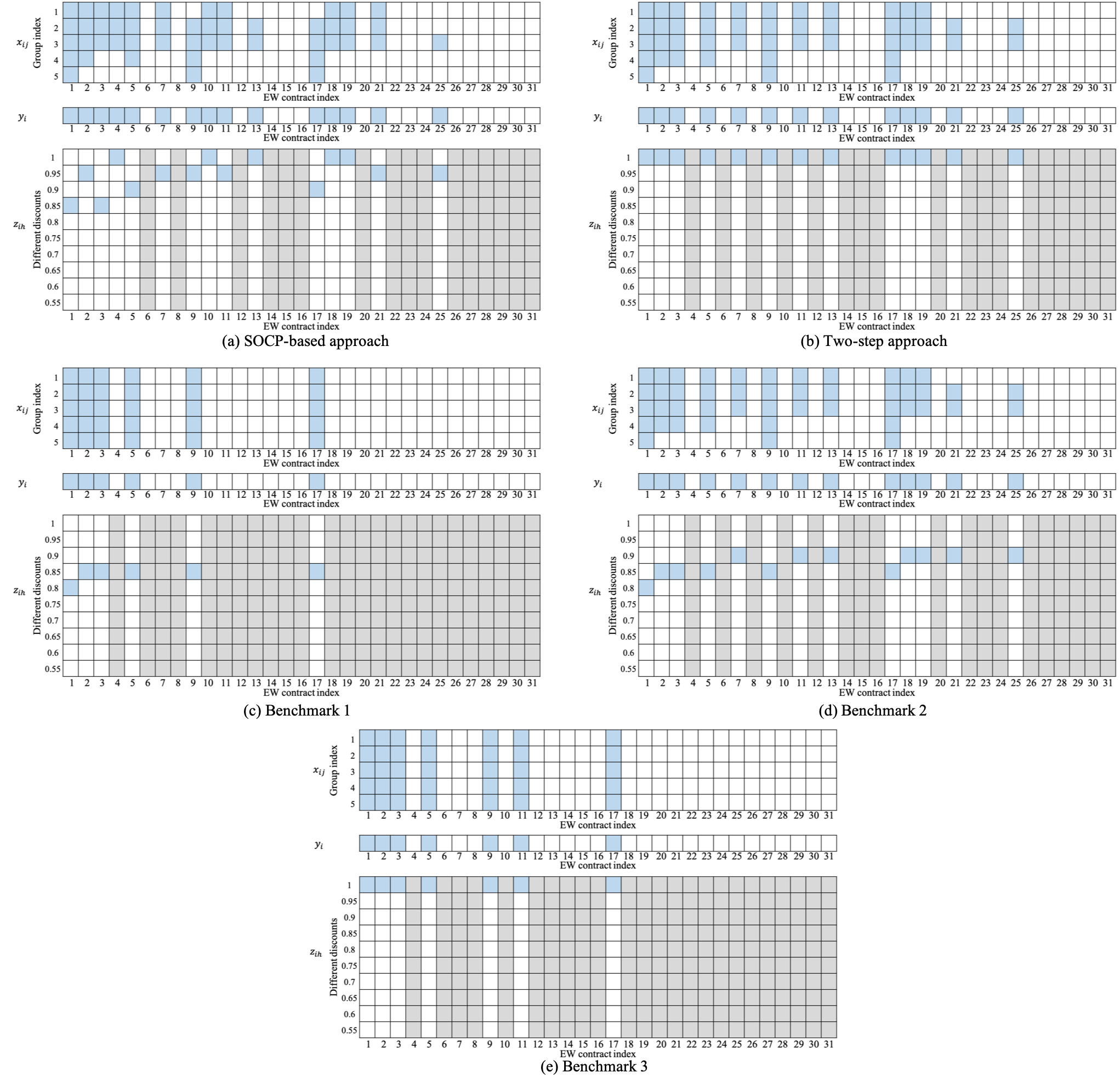}
    \caption{\centering Illustration of solutions based on different approaches}
    \label{fig:Solution}
\end{figure}

In this Section, we present the corresponding solutions based on the proposed two approaches (i.e., MISOCP reformulation and ITS approaches) and three benchmarks (i.e., \texttt{BM-1}, \texttt{BM-2} and \texttt{BM-3}) in Figure~\ref{fig:Solution}. 
The horizontal axes represent the index of EW contracts, and the vertical axes present the customer groups, advertising list, and price discount ladders for $\{x_{ij}\}$, $\{y_i\}$, and $\{z_{ih}\}$, respectively. The EW contract offer to customers and the corresponding discount selected are colored blue. Contracts that are not recommended to customers do not require pricing and are colored gray.

\subsection{Analysis on Promotion Cost}

In our optimization model in Equation (\ref{mainObject}), we still focus on the current state of third-party EW development, which is still in the early exploratory stage in practice. 
Most third-party EW providers rely primarily  on in-person sales representatives stationed at 4S dealerships to promote their warranty service, which results in substantial promotional costs.
However, with the advent of the e-commerce era for EW services, it has become significantly easier to present comprehensive information about EW contracts online, thereby reducing promotional costs, i.e., the value of $\theta$ will approach 0. 
Therefore, in this section, we evaluate the performance of our proposed approach relative to three benchmarks as the promotional cost gradually decreases, potentially approaching zero.

As shown in Table \ref{tab:PromCost}, the advent of the e-commerce epoch leads to a continuous reduction in promotional costs, thereby generating higher profits for EW providers. Meanwhile, the proposed joint optimization approach consistently delivers the best performance. Compared to scenarios with larger values of $\theta$ in early stage, the advantage of our method compared to approaches that adopt a consistent design across all automobile price bands (\texttt{BM-1} and \texttt{BM-3}) becomes more significant with the advent of e-commerce epoch. This implies that as promotional costs gradually decrease, personalized services will become more desirable.

\begin{table}[H]
  \centering
  \fontsize{8}{6}\selectfont
  \caption{Effect of promotion cost on total profit}
   {\renewcommand{\arraystretch}{1.5} 
   \begin{tabularx}{\textwidth}{>{\centering\arraybackslash}X p{0.5cm} >{\centering\arraybackslash}X>{\centering\arraybackslash}X>{\centering\arraybackslash}X>{\centering\arraybackslash}X p{0.3cm} >{\centering\arraybackslash}X>{\centering\arraybackslash}X>{\centering\arraybackslash}X}
    \toprule
    \multicolumn{1}{c}{\multirow{2}[4]{*}{$w$}} & \multirow{2}[4]{*}{$\theta$} & \multicolumn{4}{c}{Profit} &       & \multicolumn{3}{c}{Benefit (in \%)} \\
\cmidrule{3-6}\cmidrule{8-10}          &       & \texttt{Joint} & \texttt{BM-1} & \texttt{BM-2} & \texttt{BM-3} &       & \texttt{BM-1} & \texttt{BM-2} & \texttt{BM-3} \\
    \midrule
    3     & 0     & 1467.27  & 1238.47  & 1269.05  & 1407.58  &       & 18.47  & 15.62  & 4.24  \\
          & 2     & 1455.06  & 1230.98  & 1256.56  & 1400.63  &       & 18.20  & 15.80  & 3.89  \\
          & 4     & 1446.25  & 1223.59  & 1247.49  & 1393.76  &       & 18.20  & 15.93  & 3.77  \\
          & 6     & 1438.15  & 1216.31  & 1239.36  & 1387.01  &       & 18.24  & 16.04  & 3.69  \\
          & 8     & 1430.15  & 1209.10  & 1231.36  & 1380.37  &       & 18.28  & 16.14  & 3.61  \\
    4     & 0     & 2443.37  & 1902.91  & 1968.39  & 2338.54  &       & 28.40  & 24.13  & 4.48  \\
          & 2     & 2419.89  & 1892.64  & 1946.46  & 2328.20  &       & 27.86  & 24.32  & 3.94  \\
          & 4     & 2396.76  & 1882.71  & 1926.33  & 2319.67  &       & 27.30  & 24.42  & 3.32  \\
          & 6     & 2381.53  & 1872.98  & 1909.18  & 2310.15  &       & 27.15  & 24.74  & 3.09  \\
          & 8     & 2367.15  & 1863.45  & 1894.98  & 2301.87  &       & 27.03  & 24.92  & 2.84  \\
    5     & 0     & 3564.12  & 2573.75  & 2676.57  & 3351.72  &       & 38.48  & 33.16  & 6.34  \\
          & 2     & 3557.05  & 2554.29  & 2639.97  & 3355.50  &       & 39.26  & 34.74  & 6.01  \\
          & 4     & 3520.99  & 2536.90  & 2608.91  & 3325.47  &       & 38.79  & 34.96  & 5.88  \\
          & 6     & 3494.21  & 2521.17  & 2580.57  & 3315.65  &       & 38.59  & 35.40  & 5.39  \\
          & 8     & 3466.81  & 2506.47  & 2556.58  & 3300.24  &       & 38.31  & 35.60  & 5.05  \\
    \bottomrule
    \end{tabularx}}%
    
  \label{tab:PromCost}%
\end{table}%

\vspace{-2em}

\subsection{Robustness Check}\label{Sec:Rob}

\subsubsection{Robustness Check on Customer's Preference}

In this subsection, we test the robustness of our proposed methods based on different distributions of customers' attraction value ($v_{kj}$) on each subsystem and the formulation of preference weight of a EW contract ($u_{ij}$), respectively, as both of them directly influence customers' purchase probabilities as well as EW providers' profits. 

For the distribution of $v_{kj}$, besides uniform distribution we applied in Section \ref{Sec:Numerial}, we also discuss the profits and their gaps based on different optimization models when the value of $v_{kj}$ follows other distributions such as Power law and Normal distributions.

Table~\ref{tab:Sensi-dist} presents the results based on different distributions of attraction values. As shown, the joint optimization model proposed in our work consistently demonstrates significant advantages. Regardless of the distribution type, the profit improvement percentage compared to the other three benchmarks remains stable. This result effectively demonstrates the robustness of the proposed joint optimization approaches.

\begin{table}[H]
  \centering
  \fontsize{8}{6}\selectfont
  \caption{Effect of the distribution of attraction values on total profits}
 {\renewcommand{\arraystretch}{1.5} 
\begin{tabularx}{\textwidth}{>{\centering\arraybackslash}X p{2cm} >{\centering\arraybackslash}X>{\centering\arraybackslash}X>{\centering\arraybackslash}X>{\centering\arraybackslash}X p{0.5cm} >{\centering\arraybackslash}X>{\centering\arraybackslash}X>{\centering\arraybackslash}X}
    \toprule
    \multirow{2}[4]{*}{$w$} & \multirow{2}[4]{*}{Uti. Distri.} & \multicolumn{4}{c}{Profit} & \multirow{2}[4]{*}{} & \multicolumn{3}{c}{Benefit (in \%)} \\
\cmidrule{3-6}\cmidrule{8-10}          &       & \texttt{Joint} & \texttt{BM-1} & \texttt{BM-2} & \texttt{BM-3} &       & \texttt{BM-1} & \texttt{BM-2} & \texttt{BM-3} \\
    \midrule
    \multirow{3}[1]{*}{3} & Uniform & 2104.54  & 1768.63  & 1800.40 & 2038.90  &       & 19.14& 17.05 & 3.26 \\
          & Normal & 2106.08  & 1768.97  & 1800.82  & 2039.56  &         & 19.21 &  17.11& 3.30 \\
          & Power Law & 2355.43  &  1936.35 &  1982.21 & 2269.35  &       & 21.79 & 18.97 & 3.84 \\
    \multirow{3}[0]{*}{4} & Uniform & 3407.91  &  2650.01 &  2688.97 & 3308.23  &       & 28.71  & 26.84 & 3.02 \\
          & Normal & 3408.88  &  2650.62 &  2689.91  & 3309.28  &       & 28.71 & 26.84 & 3.01 \\
          & Power Law & 3742.63  & 2840.72  & 2886.65  & 3603.46  &       &  31.87 & 29.76& 3.89 \\
    \multirow{3}[1]{*}{5} & Uniform & 4972.06  &  3562.40 &  3611.09 & 4790.18  &       & 39.85 & 38.00 & 3.81 \\
          & Normal & 4972.55  & 3560.68 &  3610.79  & 4786.13  &       &39.95 &  38.04 & 3.99 \\
          & Power Law & 5409.43  & 3781.09 &  3830.93  & 5206.34  &       & 43.43 & 41.57 & 4.05 \\
    \bottomrule
    \end{tabularx}}%
  \label{tab:Sensi-dist}%
\end{table}%

In the main context of our problem, we follow the commonly adopted additive formulation in literature to evaluate the attraction value of an EW contract by adding up the corresponding attraction values of the subsystems included within it. However, in this Section, we focus on a more general setting in which the subsystem may interact (such as when functional complementarities or substitutabilities), where the additive assumption could become restrictive. Therefore, we have included discussions on extending the model of attraction values beyond the linear specification to further assess the robustness of our proposed method. Specifically, we conducted a supplemental analysis using a nonlinear formulation to describe the attraction values as follows:

\begin{equation}\label{Equ:NonUtility}
    u_{ij}=\left(1-\frac{\ln (\sum_{k \in \mathcal{W}} \mathbb{I}_{i k}) }{6}\right) \cdot \sum_{k \in \mathcal{W}} \mathbb{I}_{i k} v_{k j} - \beta_jp_{ij},
\end{equation}
for each $i\in \mathcal{N} $ and $j\in \mathcal{M}$.

Considering that the marginal attraction values may diminish as the number of subsystems increases, Equation~(\ref{Equ:NonUtility}) introduces an additional coefficient ($1 - \frac{1}{6}\ln (\sum_{k \in \mathcal{W}} \mathbb{I}_{i k})$) that related to the number of subsystems included in the EW contract and the preference weight $u_{ij}$ becomes a concave function over the number of subsystems. Based on the nonlinear preference weight formulation in Equation (\ref{Equ:NonUtility}), we present the performance of our proposed optimization approach and three benchmarks in Table \ref{tab:nonlinear1}. As shown, we find that the superiority of our joint optimization model remains significant, which is qualitatively consistent with the linear preference weight case. The robustness of our proposed approach is fully verified.

\begin{table}[H]
  \centering
  \fontsize{8}{6}\selectfont
  \caption{Performance of different approaches based on nonlinear formulations of attraction values}
  {\renewcommand{\arraystretch}{1.5} 
    \begin{tabularx}{\textwidth}{>{\centering\arraybackslash}X p{1.0cm} >{\centering\arraybackslash}X>{\centering\arraybackslash}X>{\centering\arraybackslash}X>{\centering\arraybackslash}X p{0.3cm} >{\centering\arraybackslash}X>{\centering\arraybackslash}X>{\centering\arraybackslash}X}
    \toprule
    \multirow{2}[4]{*}{$w$} & \multirow{2}[4]{*}{Case} & \multicolumn{4}{c}{Profit} &       & \multicolumn{3}{c}{Benefit (in \%)} \\
\cmidrule{3-6}\cmidrule{8-10}          &       & \texttt{Joint} & \texttt{BM-1} & \texttt{BM-2} & \texttt{BM-3} &       & \texttt{BM-1} & \texttt{BM-2} & \texttt{BM-3} \\
    \midrule
    3     & Case1 & 1987.12  & 1683.19  & 1708.84  & 1930.62  &       & 18.06  & 16.28  & 2.93  \\
          & Case2 & 670.67  & 576.02  & 585.77  & 650.25  &       & 16.43  & 14.49  & 3.14  \\
          & Case3 & 1378.51  & 1189.54  & 1210.26  & 1335.86  &       & 15.89  & 13.90  & 3.19  \\
    4     & Case1 & 3198.97  & 2557.23  & 2516.21  & 3114.91  &       & 25.10  & 27.13  & 2.70  \\
          & Case2 & 1092.88  & 910.65  & 899.94  & 1069.16  &       & 20.01  & 21.44  & 2.22  \\
          & Case3 & 2261.48  & 1876.78  & 1841.55  & 2203.54  &       & 20.50  & 22.80  & 2.63  \\
    5     & Case1 & 4576.04  & 3339.69  & 3409.49  & 4050.39  &       & 37.02  & 34.21  & 12.98  \\
          & Case2 & 1474.45  & 1188.40  & 1217.69  & 1339.90  &       & 24.07  & 21.09  & 10.04  \\
          & Case3 & 3118.90  & 2407.70  & 2478.42  & 2770.41  &       & 29.54  & 25.84  & 12.58  \\
    \bottomrule
    \end{tabularx}}%
  \label{tab:nonlinear1}%
\end{table}%

\subsubsection{Robustness Check on Failure Probability}

In Section~\ref{Subsubsec:failure}, we have discussed diversity failure probabilities for different subsystems based on the independent assumption. As In the third-party market, EW providers tends to classify automobiles into different subsystem based on their functions. For example, the core subsystems such as engine, braking systems, as well as auxiliary subsystems such as the transmission.
Different subsystems typically perform distinct functions through the coordinated operation of the components within them.
However, common causes may also exist that lead to simultaneous failures across different subsystems in practice. Therefore, in this section, we provide a extended discussion of cross-subsystem failure correlation to further verify the performance and robustness of the proposed joint optimization approach.

Specifically, for each group $j\in \{1,2,...,m\}$, we further generate a $w$-dimensional vector of correlated subsystem failure probabilities, $\boldsymbol{F}_j=[F_{1j},\ldots,F_{5j}]$ from a multivariate normal distribution $N(\boldsymbol{\mu}_j,\Sigma_j)$ with the equal mean as $\boldsymbol{\mu}_j=[0.08,\ldots,0.08]$. For the covariance matrix $\Sigma_j$, we set each diagonal element (i.e., the variance of failure probability for each subsystem) to $0.0009$, and each off-diagonal element (covariance between the failure rate of any two subsystems) to $0.00036$.
To ensure that the simulated failure probability aligns with realistic conditions, we truncate the generated data to retain only samples within the range of 0.05 to 0.2.
Based on the aforementioned setting of failure parameters, we present the profits and corresponding benefit gap in Table \ref{tab:different correlations}. 

\begin{table}[htbp]
  \fontsize{8}{8}\selectfont
  \centering
  \caption{Effect of subsystem failure probability correlation on profit}
     {
\renewcommand{\arraystretch}{1.5}
    \begin{tabularx}{\textwidth}{>{\centering\arraybackslash}X >{\centering\arraybackslash}X>{\centering\arraybackslash}X>{\centering\arraybackslash}X>{\centering\arraybackslash}X p{0.1cm} >{\centering\arraybackslash}X>{\centering\arraybackslash}X>{\centering\arraybackslash}X}
    \toprule
    \multirow{2}[4]{*}{$w$} & \multicolumn{4}{c}{Profit}    &       & \multicolumn{3}{c}{Benefit(in \%)} \\
\cmidrule{2-5}\cmidrule{7-9}          & Joint & BM-1  & BM-2  & BM-3  &       & BM-1  & BM-2  & BM-3 \\
    \midrule
    3     & 1566.09  & 1529.88  & 1269.76  & 1260.59  &       & 2.37  & 23.34  & 24.23  \\
    4     & 2863.92  & 2800.29  & 2176.35  & 2129.71  &       & 2.27  & 31.59  & 34.47  \\
    5     & 4119.53  & 3941.12  & 2902.65  & 2822.73  &       & 4.53  & 41.92  & 45.94  \\
    \bottomrule
     \end{tabularx}}%
  \label{tab:different correlations}%
\end{table}%

As shown, even though there may be correlations in failure probabilities across subsystems due to shared root causes of failure, it does not undermine the advantages of our method in EW design and pricing optimization.

\subsection{Comparison with Generic Algorithm}

As our proposed iterative two-step (ITS) solution approach also belongs to the heuristic family for solving large-scale integer programming, we adopt the Genetic Algorithm (GA), a widely used and effective heuristic method for large-scale integer programming, as another benchmark to evaluate the advantages of the ITS approach. The specific steps of the GA are outlined as follows.

\textbf{Step 1: Initialization}. Randomly generate $N_p$ individuals $(\mathbf{X}^*,\mathbf{y}^*,\mathbf{P}^*, \mathbf{Z}^*)$, and each individual is checked and repaired to satisfy the constraints in Problem (\texttt{JDPEW}).

\textbf{Step 2: Evaluation}. For each individual, compute the fitness value (objective value):
$$
\sum_{j \in \mathcal{M}} \lambda_j \sum_{i \in \mathcal{N}} q_{i j}\left(\mathbf{u}_j, \mathbf{x}_j\right)\left[p_{i j}-\sum_{k \in \mathcal{W}} \mathbb{I}_{i k} F_{k j} c_{k j}\right]-\sum_{i \in \mathcal{N}} \theta y_i.$$ 
Higher fitness values indicate better individuals.

\textbf{Step 3: Selection}. Tournament selection is used: randomly pick two individuals and select the better one (higher fitness value).

\textbf{Step 4: Crossover}. For each variable, randomly select elements from each parent with probability $\mathcal{P}_c$.

\textbf{Step 5: Mutation}. Randomly flip the variables ${y_{i}}$ and $x_{ij}$ with probability $\mathcal{P}_m$, reassign $z_{ih}$ so that $\sum_{h\in \mathcal{L}}z_{ih}=1$. After mutation, repair the individual to ensure all feasibility constraints are satisfied.

\textbf{Step 6: Elitism and Iteration}. Retain top $\tau$ percent of the total population ($N_p$ individuals) as elites for the next generation. Apply a repair step to ensure the each new generation of individuals remains feasible. Repeat evaluation, selection, crossover, and mutation for $N_{gen}$
generations until no improvement is observed or the maximum iteration limit is reached.

The detailed procedure of the GA implementation is provided in Algorithm 2 for direct comparison.

\begin{algorithm}[htbp]
\caption{Genetic Algorithm (GA)}
\KwIn{Coefficients $\theta$, $\{\lambda_j\}_{j\in \mathcal{M}}$, $\{\mathbb{I}_{ik}\}_{i \in \mathcal{N}, k \in \mathcal{W}}$, $\{P^{0}_{kj}\}_{k \in \mathcal{W}, j\in \mathcal{M}}$, $\{d_h\}_{h\in \mathcal{L}}$, $\{F_{kj}\}_{k\in \mathcal{W}, j\in \mathcal{M}}$, $\{c_{jk}\}_{j\in \mathcal{M},k\in \mathcal{W}}$,  $N_p$, $N_{gen}$, $\tau$, $\mathcal{P}_c$, $\mathcal{P}_m$}
\KwOut{Best feasible solution $(\mathbf{X}^*,\mathbf{y}^*,\mathbf{P}^*, \mathbf{Z}^*)$ and the corresponding objective value}
\BlankLine
\textbf{Initialization:}\\
Generate $N_p$ individuals for each element in $(\mathbf{X}, \mathbf{y}, \mathbf{P},\mathbf{Z})$ randomly; \\
each individual is checked and repaired to satisfy the constraints.
\BlankLine
\For{$g = 1$ to $N_{gen}$}{
  \textbf{Evaluation:}\\
  \For{each individual}{
    Compute $p_{ij}$ and $q_{ij}$;\\
    Evaluate objective $f =
\sum_{j \in \mathcal{M}} \lambda_j \sum_{i \in \mathcal{N}} q_{i j}\left(\mathbf{u}_j, \mathbf{x}_j\right)\left[p_{i j}-\sum_{k \in \mathcal{W}} \mathbb{I}_{i k} F_{k j} c_{k j}\right]-\sum_{i \in \mathcal{N}} \theta y_i$,\\
  }
  \textbf{Selection:} Use tournament selection to choose parents.\\
  \textbf{Crossover:} Combine parents with probability $\mathcal{P}_c$.\\
  \textbf{Mutation:} Randomly flip bits of elements of $\mathbf{X}$, $\mathbf{y}$, and $\mathbf{Z}$, i.e., $\{x_{ij}\}$, $\{y_i\}$, and $\{z_{ih}\}$, with probability $\mathcal{P}_m$, and calculate the corresponding $\{p_{ij}\}$ in $\mathbf{P}$. \\
  Each new generation of individuals is checked and repaired to satisfy the constraints.\\
  Retain top $\tau\%$ of $N_p$ elite individuals.\\
}
\end{algorithm}

Based on the aforementioned setting of GA, we compare its performance with our proposed ITS approach in Table~\ref{tab:CompareGA} related to profit and computation time. The GA is configured with generation limit $N_{gen} = 80$, population size $N_p= 60$, crossover probability $\mathcal{P}_c= 0.5$, mutation probability $\mathcal{P}_m= 0.12$, and an elite proportion $\tau=0.05$.
As shown, although the GA requires relatively less computation time, the profits it generates remain consistently and significantly lower than those achieved by our proposed method, amounting to approximately 96–98\% of the profits obtained by our approach. 
Moreover, as the problem scale increases with larger values of $w$, the profit gap between the GA and our approach expands.

The advantage of the ITS approach lies in the partition of the decision variables in the original integer programming problem into two groups (related to EW contract design and EW pricing, respectively). In each iteration, one group of variables is fixed, and the resulting subproblem is solved to obtain the optimal solution for the remaining group. This procedure is repeated until the overall solution converges. Although convergence to the optimal solution cannot be guaranteed, each subproblem solved within an iteration yields its own optimal solution, leading to higher profits.

In practical decision-making, the EW provider is not required to make real-time decisions, making a computation time of several minutes acceptable.
However, profit performance is of greater importance, and a 2–4\% loss in profit is typically unacceptable. These results highlight the managerial relevance and effectiveness of our proposed ITS solution approach.

\begin{table}[htbp]
  \centering
  \fontsize{8}{8}\selectfont
  \caption{Performance comparison with genetic algorithm}
     {
\renewcommand{\arraystretch}{1.5}
    \begin{tabularx}{\textwidth}{>{\centering\arraybackslash}X >{\centering\arraybackslash}X>{\centering\arraybackslash}X p{0.1cm} >{\centering\arraybackslash}X >{\centering\arraybackslash}X>{\centering\arraybackslash}X}
    \toprule
    \multirow{2}[4]{*}{$w$} & \multicolumn{2}{c}{\textbf{Profit}} &     & \multicolumn{2}{c}{\textbf{Time (s)}}   \\
\cmidrule{2-3}\cmidrule{5-6}          & \textbf{ITS} & \textbf{GA} &        & \textbf{ITS} & \textbf{GA} \\
    \midrule
    3      & 2106.33  & 2067.88  &         & 1.72  & 2.77  \\
    4       & 3405.35  & 3313.52  &         & 195.31  & 4.38  \\
    5      & 4908.67  & 4740.03  &         & 693.35  & 6.14  \\
    \bottomrule
    \end{tabularx}}%
  \label{tab:CompareGA}%
\end{table}%

\end{spacing}

\end{document}